\documentclass[onefignum,onetabnum]{siamonline250211}

\usepackage{lipsum}
\usepackage{amsfonts}
\usepackage{graphicx}
\usepackage{epstopdf}
\usepackage{xcolor}
\usepackage{algorithmic}
\usepackage{bm}
\usepackage[textsize=scriptsize,backgroundcolor=white]{todonotes}
\ifpdf
\numberwithin{equation}{section}

  \DeclareGraphicsExtensions{.eps,.pdf,.png,.jpg}
\else
  \DeclareGraphicsExtensions{.eps}
\fi

\usepackage[shortlabels]{enumitem}
\setlist[enumerate]{leftmargin=.5in}
\setlist[itemize]{leftmargin=.5in}

\newsiamremark{remark}{Remark}
\newsiamremark{hypothesis}{Hypothesis}
\crefname{hypothesis}{Hypothesis}{Hypotheses}
\newsiamthm{claim}{Claim}

\headers{Learning Dynamical Systems with the Spectral Exterior Calculus}{S. Das, D. Giannakis, Y. Gu, and J. Slawinska}

\title{Learning Dynamical Systems with the Spectral Exterior Calculus\thanks{Submitted to the editors DATE.
\funding{Dimitrios Giannakis acknowledges support from the U.S.\ Department of Defense, Basic Research Office under Vannevar Bush Faculty Fellowship grant N00014-21-1-2946 and the U.S.\ Office of Naval Research under MURI grant N00014-19-1-242. Yanbing Gu was supported as a PhD student from these grants. Dimitrios Giannakis and Joanna Slawinska acknowledge support from the U.S.\ Department of Energy grant DE-SC0025101.}}}

\author{Suddhasattwa Das\thanks{Department of Mathematics and Statistics, Texas Tech University, Lubbock, TX} (\email{suddas@ttu.edu})
  \and Dimitrios Giannakis\thanks{Department of Mathematics, Dartmouth College, Hanover, NH} \email{dimitrios.giannakis@dartmouth.edu}, \email{yanbing.gu.gr@dartmouth.edu}, \email{joanna.m.slawinska@dartmouth.edu}) \and Yanbing Gu\footnotemark[3]
\and Joanna Slawinska\footnotemark[3]}

\usepackage{amsopn}

\ifpdf
\hypersetup{
  pdftitle={Learning Dynamical Systems with the Spectral Exterior Calculus},
  pdfauthor={S. Das, D. Giannakis, Y. Gu, and J. Slawinska}
}
\fi

\usepackage{amssymb}
\DeclareMathOperator{\spn}{span}
\DeclareMathOperator{\ran}{ran}
\DeclareMathOperator{\tr}{tr}

\DeclareMathOperator{\dom}{dom}
\DeclareMathOperator{\rank}{rank}
\DeclareMathOperator{\diag}{diag}
\DeclareMathOperator{\divr}{div}

\newcommand{\calM}{ \mathcal{M} }
\newcommand{\Hodge}{ \mathbb{H} }
\newcommand{\real}{\mathbb{R}}

\newcommand{\num}{\mathbb{N}}
\newcommand{\norm}[1]{\left\| #1 \right\|}

\newcommand{\paran}[1]{\left( #1 \right)}

\newcommand{\grad}{\nabla}
\newcommand{\metric}{\mathfrak{g}}
\renewcommand{\qed}{$\blacksquare$}
\newcommand{\LQ}{J}

\newcommand{\LD}{L_D}

\newtheorem{Assumption}{Assumption}
\crefname{Assumption}{Assumption}{Assumptions}
\crefname{app}{Appendix}{Appendices}
\Crefname{app}{Appendix}{Appendices}

\begin{document}

\maketitle

\begin{abstract}
We present a data-driven framework for learning dynamical systems on compact Riemannian manifolds based on the spectral exterior calculus (SEC). This approach represents vector fields as linear combinations of frame elements constructed using the eigenfunctions of the Laplacian on smooth functions, along with their gradients. Such reconstructed vector fields generate dynamical flows that consistently approximate the true system, while being compatible with the nonlinear geometry of the manifold. The data-driven implementation of this framework utilizes embedded data points and tangent vectors as training data, along with a graph-theoretic approximation of the Laplacian. In this paper, we prove the convergence of the SEC-based reconstruction in the limit of large data. Moreover, we illustrate the approach numerically with applications to dynamical systems on the unit circle and the 2-torus. In these examples, the reconstructed vector fields compare well with the true vector fields, in terms of both pointwise estimates and generation of orbits.
\end{abstract}



\section{Introduction}
\label{sec:introduction}

Given a dynamical system
\begin{displaymath}
    \dot{x}(t) = V\rvert_{x(t)}
\end{displaymath}
generated by a vector field $V\colon \mathcal M \to T \mathcal M$ on a manifold $\mathcal M$, a problem of considerable practical importance is to approximate (learn) the dynamics using samples of the state $x(t) \in \mathcal M$, sometimes indirectly via measurements. Over the years, a large number of parametric and non-parametric data-driven techniques have been developed to perform this task.

Parametric methods seek an optimal approximation to the generating vector field $V$ of the dynamics within a parameterized family. In essence, the choice of parametric family amounts to assuming a specific form for the dynamical laws, and the learning task is mainly to determine parameters within the prescribed class. Parametric methods are particularly effective in applications with significant prior knowledge about the governing laws (which can constrain the class of the learned model), and may be combined with Bayesian inference techniques \cite{Stuart10} to yield uncertainty-quantified parameter estimates. Some examples of the utility of parametric methods in specific domains are in the modeling of epidemics \cite{bailey1975biomath, Brauer2012biomath} and forecasting and estimation of various aspects of geophysical fluid dynamics \cite[e.g.]{PalmerHagedorn06, MajdaWang06}.

Our focus in this paper will be on non-parametric methods. These approaches do not rely on any foreknowledge about the explicit format of the dynamical laws, and instead leverage data to approximate $V$ within a formally infinite-dimensional hypothesis space. This provides higher flexibility, sometimes with universal approximation guarantees \cite{HornikEtAl89,MicchelliEtAl06}, which may overcome structural model errors associated with parametric approaches. At the same time, non-parametric methods tend to have higher training data requirements stemming from high dimensionality of the hypothesis spaces and ambient data spaces involved in real-world applications. Another major challenge is to ensure that the methods are refinable, i.e., they produce asymptotically consistent results with increasing amounts of training data.

There have been several approaches to non-parametric estimation of dynamical systems. They have diverse mathematical and computational foundations, including Fourier averaging \cite{LangeEtAl2021, DasJim2017_SuperC}, non-parametric regression \cite{Lin2004statistical, HallReimannRice2000, Silverman1984, TompkinsRamos2020,AlexanderGiannakis20,HamziOwhadi21,KoltaiKunde23}, neural network approximation \cite{Gonzalez-GarciaEtAl98, ZhuZabaras18, YeungEtAl2019, BhattacharyaEtAl21, HarlimEtAl2021, MaEtAl_2018, Maulik_EtAl_2020, ParkEtAl24}, reservoir computing \cite{JaegerHaas2004, GrigoryevaHartOrtega2021, GononOrtegafading_fading_2021}, Markov state models \cite{vidyasagar2005rlze, Das2024zero, finesso2010approx}, and linear evolution operators \cite{Penland89,DellnitzJunge99,SchutteEtAl01,Mezic05}. Sparse identification of nonlinear dynamics (SINDy) \cite{BruntonEtAl16} incorporates sparsity-inducing penalty terms in regression models to yield approximations of the dynamical vector field  which are sparse in a user-supplied dictionary of functions. Physics-informed neural networks (PINNSs) \cite{RaissiEtAl19} and neural operators \cite{KovachkiEtAl23} are neural-network based approaches that have been successful in learning infinite-dimensional systems governed by partial differential equations. Meanwhile, methods based on Koopman and transfer operators leverage the intrinsically linear action of a dynamical flow on vector spaces of observables to build linear models for the evolution of the components of the state vector (treated in this context as observables) \cite{RowleyEtAl09,WilliamsEtAl15,BerryEtAl15,Kawahara16,FroylandKoltai17,BruntonEtAl17, DasEtAl21,KlusEtAl20b,KosticEtAl22,ColbrookTownsend24}. These methods can be seen in a similar vein as the lift-and-learn approach \cite{QianEtAl20}, which seeks to embed the nonlinear state space evolution to a higher-dimensional space (e.g., a space of observables or measures) in which dynamical nonlinearities are brought into a standard form, or eliminated altogether. Recent work \cite{ColbrookEtAl24b} has studied approximation barriers faced by data-driven techniques utilizing spectral decompositions of Koopman operators.

Many of the techniques outlined above effectively access the vector field $V$ through its representation as a vector-valued function $\vec V\colon \mathcal M \to \mathbb R^d$ taking values in a Euclidean space $\mathbb R^d$ in which the state space manifold $\mathcal M$ is embedded. While every sufficiently smooth embedding $F \colon \mathcal M \hookrightarrow \mathbb R^d$ induces a vector-valued function $\vec V = F_* V $ through the pushforward map $F_*$ on vector fields, such a representation is not intrinsic to the dynamical system as it depends on the embedding map. To put it differently, while it is true that every vector field $V$ induces an $\mathbb R^d$-valued function $\vec V$ under an embedding, not every $\mathbb R^d$-valued function on $\mathcal M$ is realized by the pushforward of a vector field. Correspondingly, learning $\vec V$ in a hypothesis space of $\mathbb R^d$-valued functions carries a risk of producing structurally inconsistent results.

\subsection{Our approach}\label{sec:approach}

We take a distinct path of learning the generator of the dynamics in its intrinsic geometric and operator-theoretic forms:
\begin{itemize}
    \item Geometrically, the dynamical vector field corresponds to a section $V\colon \mathcal M \to T \mathcal M$ of the tangent bundle of $\mathcal M$.
    \item From an operator-theoretic point of view, $V$ acts as a linear operator on the space of continuously differentiable, real-valued functions on $\mathcal M$,
        \begin{equation}
            \label{eqn:vector_field_d}
            Vf = \langle V, df \rangle, \quad \forall f \in C^1(\mathcal M),
        \end{equation}
        obeying the Leibniz rule,
        \begin{displaymath}
            V(fg) = (V f) g + f(V g), \quad \forall f, g \in C^1(\mathcal M).
        \end{displaymath}
\end{itemize}

Here, $df\colon \mathcal M \to T^*\mathcal M $ is the exterior derivative or differential of $f$ (and therefore a covector field), and $\langle \cdot, \cdot \rangle$ denotes the natural pairing between tangent vectors and dual vectors on $\mathcal M$. When $\mathcal M$ is equipped with a Riemannian metric $\metric$, we equivalently have
\begin{equation}
\label{eqn:def:vectorfield}
    Vf = V \cdot  \grad f \equiv \left\langle V, \grad f \right\rangle_{\metric},
\end{equation}
where $\grad$ and $\langle \cdot, \cdot\rangle_{\metric}$ are the gradient operator and Riemannian inner product on tangent vectors induced by $\metric$, respectively.

In more detail, our approach is to learn $V$ in a hypothesis space spanned by a dictionary of intrinsic vector fields. The training data consists of samples  $\{y_n\}_{n=1}^{N} \subset \real^d$ which are the images $y_n = F(x_n)$ of
points $\{x_n\}_{n=1}^{N}\subset \calM$ under an unknown embedding $F \colon \mathcal M \to \mathbb R^d$, along with samples $\vec V(x_n) = F_* V\rvert_{x_n}$ (``arrows'') of the corresponding pushforward vector field. We build this hypothesis space using the Spectral Exterior Calculus (SEC) \cite{BerryGiannakis20}---a data-driven technique that derives frames (overcomplete bases) of Hilbert spaces of vector fields and differential forms on Riemannian manifolds using spectral data of the Laplace--Beltrami operator $\Delta = - \divr \circ \grad$ on functions. This leads to an approximation of the form
\begin{equation}
    \label{eq:sec_approx_intro}
    V \approx V^{(\ell)} = \sum_{ij} b_{ij} B_{ij}, \quad b_{ij} \in \mathbb R, \quad B_{ij} = \phi_i \grad \phi_j,
\end{equation}
where $B_{ij}\colon \mathcal M \to T \mathcal M$ are frame elements given by products of Laplace--Beltrami eigenfunctions $\phi_i \in C^\infty(\mathcal M)$ and their gradients, and $b_{ij}$ are the corresponding expansion coefficients. In~\cref{eq:sec_approx_intro}, the $^{(\ell)}$ superscript collectively represents various approximation parameters (such as the number of frame elements employed), which will be made precise in subsequent sections.

To build the frame elements $B_{ij}$, the SEC leverages (i) the \emph{carr\'e du champ} identity \cite{BakryEtAl14} obeyed by the Laplacian,
\begin{equation}
    \label{eqn:carre_du_champ}
    \Delta(fg) = (\Delta f) g + f (\Delta g) - 2 \grad f \cdot \grad g, \quad \forall f, g \in C^\infty(\mathcal M),
\end{equation}
which allows computing Riemannian inner products between gradient vector fields from the action of $\Delta$ on functions; (ii) the $C^\infty(\mathcal M)$-module structure of spaces of vector fields, which allows building generating sets of vector fields from products of the form $f_i \grad f_j$ from a sufficiently rich collection of functions $f_i \in C^\infty(\mathcal M)$; and (iii) graph-theoretic algorithms (e.g., \cite{CoifmanLafon2006,TrillosEtAl_error_2020,HeinEtAl07,WormellReich21}) for approximation of the Laplacian induced by the embedding $F$ in data space using the dataset $ \{ y_n \}_{n=1}^N$. Furthermore, the SEC provides a representation of vector fields as operators on functions, allowing approximation of the coordinate representation of $\vec V$ in the embedding space,
\begin{displaymath}
    \vec V(y) \approx \vec V^{(\ell)}(y) = F_* V^{(\ell)}(y) \equiv V^{(\ell)} F(x), \quad y = F(x),
\end{displaymath}
where $V^{(\ell)}$ acts on $F$ componentwise as a linear operator.

In this approximation, the components of $\vec V^{(\ell)}$ lie in a space of $C^1$ functions created by a kernel integral operator on $\mathcal M$ induced by the training data. The vector-valued function $\vec V^{(\ell)}$ thus generates a smooth dynamical system on $\mathbb R^d$,
\begin{equation}
    \label{eqn:vecivp}
    \dot y(t) = \vec V^{(\ell)}\rvert_{y(t)},
\end{equation}
that was learned in a hypothesis space of intrinsic vector fields on the manifold $\mathcal M$. See \cref{fig:Time_complexity} for a schematic overview of the training and prediction phases of our approach.

\begin{figure}
    \centering
    \includegraphics[width=\textwidth]{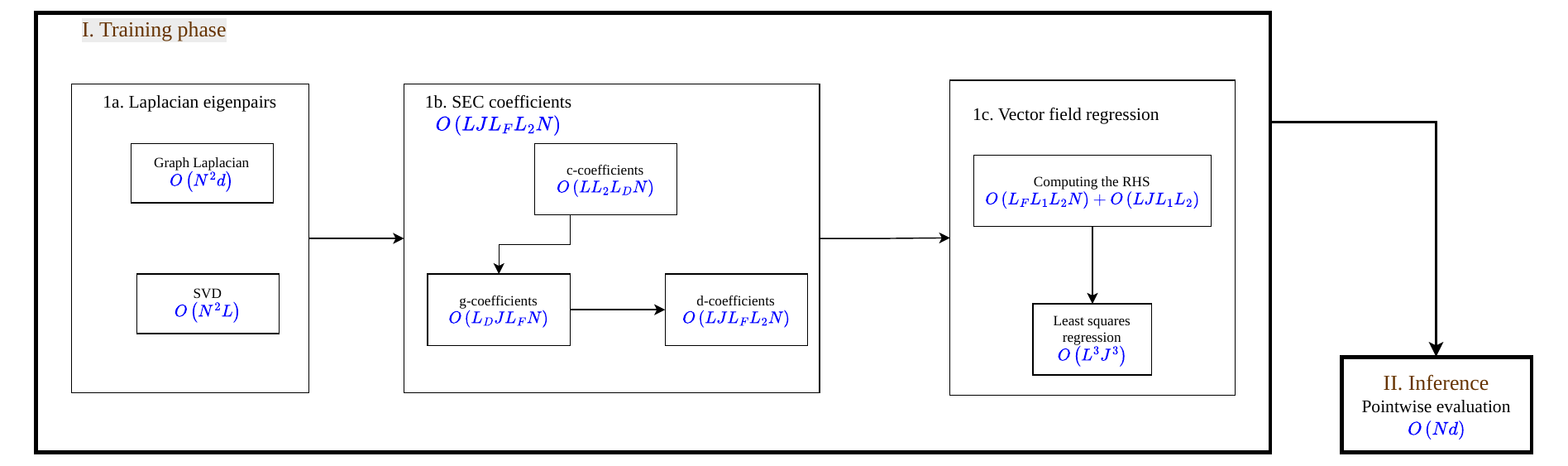}
    \caption{Schematic overview of the training and prediction (inference) phases of our numerical procedure. The training phase has three parts: eigendecomposition of the graph Laplacian (1a), computation of SEC coefficients (1b), and solution of the vector field regression problem (1c). Note that parts~1a and~1b use only the spatial locations of the training data on the manifold, and part~1c uses both locations and vector field data. As a result, parts~1a and~1b may be skipped if the vector field is changed but the sampled points on the manifold are not. The boxes show numerical time complexity estimates based on brute-force linear algebra computations. The actual cost may be improved if specialized numerical techniques are used such as low-rank kernel operator approximations in the inference phase.}
    \label{fig:Time_complexity}
\end{figure}

Our main result, \cref{thm:1}, establishes the precise nature of the convergence of the SEC-based approximation of vector fields described in this paper. The statement of the theorem makes some general assumptions on the vector field $V$ and manifold $\calM$ supporting it (\cref{A:1}), the embedding of the manifold in Euclidean space (\cref{A:2}), and the nature of its sampling (\cref{A:3}). The numerical approximations employed are based on kernel methods that make some general assumptions on the kernel (\cref{A:4}), but otherwise do not assume prior knowledge of the structure of the manifold, the dynamical vector field, or the sampling distribution of the training data. In \cref{prop:L2_vector_approx}, we give an error estimate for dynamical orbits (solutions of~\cref{eqn:vecivp}) generated by vector field approximations (including but not limited to SEC-based approximations) that converge in a Hodge norm employed in \cref{thm:1}.

We test our approach with numerical experiments involving dynamical systems on manifolds embedded in Euclidean space, including rotation systems on circles and tori and a Stepanoff flow on the 2-torus \cite{Oxtoby53}. The examples demonstrate accurate reconstruction of the dynamical vector field on the embedded manifold as well as stable out-of-sample evaluation away from the manifold.  We also show comparisons between the solutions to initial-value problems associated with the true vector fields and the SEC approximations. The results illustrate the efficacy of the SEC in reproducing a variety of dynamical behaviors, including metastable fixed points and periodic, quasiperiodic, and aperiodic orbits.

\subsection{Plan of the paper}

In \cref{sec:probstatement}, we introduce the class of systems under study and establish basic notation. In \cref{sec:dynamicslearn}, we describe our SEC-based scheme for approximation of the dynamical vector field assuming prior knowledge of the spectrum of the Laplace--Beltrami operator and its eigenfunctions. In \cref{sec:data_driven}, we present a data-driven formulation of that scheme, which casts the approximation framework of \cref{sec:dynamicslearn} in a basis of approximate Laplace--Beltrami eigenfunctions learned using graph-theoretic techniques. We analyze the convergence properties of these approximations in \cref{sec:convanalysis}. \Cref{sec:numexp} presents our numerical experiments. Our primary conclusions and a discussion on future directions are included in \cref{sec:discussion}. The paper contains appendices on numerical methods (\cref{app:algorithms}) and the proof of \cref{prop:L2_vector_approx} (\cref{app:proof_orbit_recon}). A Python implementation of our approach, reproducing the numerical results presented in this paper, can be found at the repository \url{https://github.com/ygu626/SEC}.

\section{Preliminaries and notation}
\label{sec:probstatement}

We consider a dynamical flow $\Phi^t \colon \mathcal M \to \mathcal M$, $t \in \mathbb R$, on a differentiable manifold $\mathcal M$, generated by a vector field $V \colon \mathcal M \to T \mathcal M $. The system is observed by means of a map $F \colon \mathcal M \to \mathbb R^d$ into a $d$-dimensional Euclidean data space. We shall make the following standing assumptions on the manifold, dynamical system, and observation map.

\begin{Assumption} \label{A:1}
    The manifold $\mathcal M$ is closed (i.e., smooth, compact, and boundaryless) and orientable, and the vector field $V$ is continuous.
\end{Assumption}

\begin{Assumption} \label{A:2}
    The observation map $F \colon \calM \to \real^d$ is a $C^\infty$ embedding.
\end{Assumption}

We will let $F_{*x} \colon T_x \mathcal M \to T_{F(x)} \mathbb R^d \cong \mathbb R^d$ denote the pushforward map on tangent vectors at $x \in \mathcal M$ into Euclidean vectors in data space, and $F_* \colon T \mathcal M \to T \mathbb R^d \cong \mathbb R^{2d}$ the corresponding map on the tangent bundle of $\mathcal M$. We will also let $Y = F(\mathcal M) \subset \mathbb R^d$ denote the image of the manifold in data space under the embedding $F$.

Within this setup, our training data consists of:
\begin{enumerate}
    \item Samples $ \{ y_n \in \mathbb R^d \}_{n=1}^N$ in data space given by the images $y_n = F(x_n)$ of a collection of points $ \{ x_n \}_{n=1}^N$ on the manifold for some $N \in \mathbb N$.
    \item Corresponding samples $ \{ \vec v_n \in \mathbb R^d \}_{n=1}^N$ of the pushforward of the dynamical vector field in data space, where $\vec v_n = F_{*x_n} V(x_n)$.
\end{enumerate}
Note that we do not assume any explicit knowledge of the manifold points $x_n$, the vector field $V$, or the observation map $F$. In what follows, $C^r(\mathcal M, T \mathcal M)$, $ r\in \mathbb N_0$, will denote the space of $r$-times continuously differentiable vector fields on $\mathcal M$, i.e., the space of $r$-times continuously differentiable sections of the tangent bundle $T \mathcal M$.

Let $\metric$ be the Riemannian metric tensor on $\calM$ induced by the embedding $F$, i.e., $\metric_x(u, v) = F_{*x} u \cdot F_{*x} v$ for $u, v \in T_x \mathcal M$, where $\cdot \colon \mathbb R^d \times \mathbb R^d \to \mathbb R$ is the standard Euclidean inner product. The metric $\metric$ induces a gradient operator $\nabla \colon C^1(\mathcal M) \to C(\mathcal M, T \mathcal M)$, a finite volume measure $\nu$ on $\calM$, and a divergence operator $\divr \colon C^1(\mathcal M, T \mathcal M) \to C(\mathcal M)$ such that
\begin{displaymath}
    \int_{\mathcal M} f \divr W \, d\nu = - \int_{\mathcal M}  \left\langle \nabla f(x), W(x) \right\rangle_{\metric(x)} \, d\nu(x), \quad \forall f \in C^1(\mathcal M), \quad \forall W \in C^1(\mathcal M, T \mathcal M).
\end{displaymath}
The Laplace--Beltrami operator $\Delta \colon C^2(\mathcal M) \to C(\mathcal M)$ induced by $\metric$ is defined as $\Delta = - \divr \circ \nabla$, and has an associated orthonormal basis $ \{ \phi_j \}_{j \in \mathbb N_0}$ of $L^2(\nu)$ consisting of eigenfunctions,
\begin{equation}
    \label{eqn:lapl_eig}
    \Delta \phi_j = \lambda_j \phi_j, \quad \phi_j \in C^\infty(\mathcal M), \quad \mathbb \lambda_j \geq 0,
\end{equation}
where the corresponding eigenvalues are ordered as $0 = \lambda_0 < \lambda_1 \leq \lambda_2 \leq \cdots \nearrow \infty$. In what follows, $\varphi_j \in C^\infty(\real^d)$ will be extensions of the eigenfunctions $\phi_j$; i.e., $\varphi_j(y) = \phi_j(x)$ for $y = F(x)$. In \cref{subsec:kernel}, we will show how these extensions $\varphi_j$ are constructed using kernel integral operators.

Next, we consider the sampling distribution of our training data.

\begin{Assumption} \label{A:3}
    The points $x_1, x_2, \ldots \in \mathcal M$ underlying our training data lie in an equidistributed sequence with respect to a probability measure $\mu$ with a strictly positive density $\sigma := \frac{d\mu}{d\nu} \in C^\infty(\mathcal M)$; that is, $\lim_{N\to\infty} \frac{1}{N} \sum_{n=1}^N f(x_n) = \int_{\mathcal M} f \sigma\, d\nu$, for all $f \in C(\mathcal M)$. 
\end{Assumption}

Let $\delta_x$ denote the Dirac $\delta$-measure supported at $x \in \mathcal M$. \Cref{A:3} is equivalent to weak convergence of the empirical sampling measures $\mu_N := \frac{1}{N}\sum_{n=1}^N\delta_{x_n}$ to $\mu$. For the rest of the paper, $\nu_* := F_* \nu$ will be the pushforward of the volume measure of the manifold to data space $\mathbb R^d$. We also define the Hilbert space $L^2(\nu_*; \mathbb R^d)$ of equivalence classes of $\mathbb R^d$-valued functions on data space $\mathbb R^d$ with respect to $\nu_*$ to have inner product $\langle f, g\rangle_{L^2(\nu_*; \mathbb R^d)} = \int_{\mathbb R^d} f(y) \cdot g(y) \, d\nu_*(y)$. Note that this space includes the pushforwards of elements of $\mathbb H$ under $F$, as well as other square-integrable $\mathbb R^d$-valued functions which are not tangent to the embedded manifold $Y$.

\section{Vector field approximation framework}
\label{sec:dynamicslearn}

In this section, we describe the analytical aspects of our approach for approximating dynamical vector fields using the SEC.

\subsection{Frames for Hilbert spaces of vector fields} \label{subsec:V_op}

We equip the space of continuous vector fields $C(\mathcal M, T \mathcal M)$ with the Hodge inner product
\begin{equation} \label{eqn:calH}
    \left\langle V, W \right\rangle_{\Hodge} := \int_{\calM} \left\langle V(x), W(x) \right\rangle_{\metric(x)} \, d\nu(x).
\end{equation}
This is finite for all $V, W \in C(\mathcal M, T \mathcal M)$ by compactness of $\mathcal M$. Taking the closure with respect to the norm $\lVert \cdot \rVert_{\Hodge}$ induced by this inner product leads to a Hilbert space of vector fields which we denote by $\Hodge$. Equivalently, elements $V \in \Hodge$ may be characterized by the requirement that the $\nu$-a.e.\ defined map $x \mapsto \lVert V(x)\rVert_{\metric_x}$ lies in $L^2(\nu)$.

As noted in \cref{sec:introduction}, every $V \in C(\mathcal M, T \mathcal M)$ acts as linear operator $V \colon C^1(\mathcal M) \to C(\mathcal M)$ on continuously differentiable functions on the manifold via~\cref{eqn:vector_field_d}, or, equivalently, \cref{eqn:def:vectorfield}. In what follows, we will sometimes use the notation $V \triangleright f \equiv V f$ to distinguish application of $V$ as an operator on a function $f \in C^1(\mathcal M)$ from evaluation $V(x) \in T_x \mathcal M$ as a section $V \colon \mathcal M \to T \mathcal M$ of the tangent bundle at a point $x \in \mathcal M$.

Over the years, a variety of numerical techniques for approximation of vector fields (as well as more general tensor fields) on manifolds have been developed, including methods based on simplicial complexes \cite{Hirani2003, Desbrunetal2008}, finite elements \cite{Arnoldetal2006}, wavelets \cite{Lessig01}, and radial basis functions \cite{HarlimEtAl24}. The SEC scheme \cite{BerryGiannakis20} employed in this work performs vector field approximation in \emph{frames} for the Hilbert space of vector fields $\Hodge$, constructed from eigenfunctions of the Laplace--Beltrami operator $\Delta$. We first recall the definition of a frame of a Hilbert space \cite{Mallat09}.

\begin{definition}
    \label{def:frame}
    A countable subset $ \{ u_i \}_{i \in \mathbb I}$ of a Hilbert space $H$ is said to be a frame if there exist constants $C_1, C_2 >0$ (called frame constants) such that
    \begin{equation}
        \label{eqn:frame_bound}
       C_1 \lVert  v \rVert_H^2 \leq \sum_{i \in \mathbb I} \lvert \langle u_i, v\rangle_H\rvert^2 \leq C_2 \lVert v \rVert_H^2, \quad \forall v \in H.
    \end{equation}
\end{definition}

Intuitively, a frame is a generating set for $H$ (in the sense of the lower bound in \cref{def:frame}) that is not ``too overcomplete'' (in the sense of the upper bound). In particular, the frame condition~\cref{eqn:frame_bound} implies:
\begin{itemize}
    \item The operator $T \colon H \to \ell^2(\mathbb I)$ where $T v = (\langle u_i, v\rangle_H)_{i \in \mathbb I}$, called \emph{analysis operator} is bounded, injective, and has closed range.
    \item The adjoint $T^* \colon \ell^2(\mathbb I) \to H$, called \emph{synthesis operator}, is surjective (and bounded).
    \item $S := T^* T$, called \emph{frame operator}, is a bounded operator on $H$ with bounded inverse.
    \item $G := T T^*$, called \emph{Gram operator}, is a bounded operator on $\ell^2(\mathbb I)$ with closed range, and thus with bounded pseudoinverse, $G^+ = T S^{-2} T^*$.
\end{itemize}
Using these operators, we can express any vector $v \in H$ as the following linear combination of frame elements,
\begin{equation}
    \label{eq:frame_expansion}
    v = T^* G^+ T v = \sum_{i \in \mathbb I} v_i u_i, \quad v_i = \langle S^{-1} u_i, v\rangle_H.
\end{equation}
While, in general, this expansion is not unique, the sequence of coefficients $v_i$ in~\cref{eq:frame_expansion} has the minimal $\ell^2$ norm among all sequences $(\tilde v_i)_i \in \ell^2(\mathbb I)$ such that $v = \sum_{i \in \mathbb I} \tilde v_i u_i$. Note that if $ \{ u_i \}_{i \in \mathbb I}$ is an orthonormal basis of $H$ then $S$ and $G$ are equal to identity operators and~\cref{eq:frame_expansion} reduces to a standard orthonormal basis expansion.

Returning to our SEC example, we build frames for $\Hodge$ that consist of vector fields of the form
\begin{equation} \label{eqn:def:Bpq}
    B_{i j} := \phi_i \grad \phi_j \in C(\mathcal M, T \mathcal M), \quad i \in \num_0, \quad j \in \mathbb N.
\end{equation}
The following result \cite[Theorem~4.3]{BerryGiannakis20} guarantees that such fields form a frame so long as a sufficiently large, but finite,  number of gradient fields $\nabla \phi_j$ are employed.
\begin{theorem} \label{lem:frame}
    Under \crefrange{A:1}{A:3}, there exists an integer $\LQ \geq \dim(\calM)$ such that for every $J' \geq \LQ$, the collection $\left\{ B_{i j} \;:\; i\in \num_0, \ 1\leq j\leq J' \right\}$ is a frame for $\Hodge$.
\end{theorem}

\begin{remark}
    A motivation of the frame element construction in~\cref{eqn:def:Bpq} is that the space $C(\mathcal M, T \mathcal M)$ of continuous vector fields on a closed smooth manifold is a finitely generated $C(\mathcal M)$-module. That is, there exists a finite collection of vector fields $W_1, \ldots, W_{\LQ } \in C(\mathcal M, T \mathcal M) $ such that $V = \sum_{j=1}^{\LQ } f_j W_j$ for every vector field $V \in C(\mathcal M, T \mathcal M) $ for a  collection $f_1, \ldots, f_{\LQ } \in C(\mathcal M)$ of continuous functions. By results from differential geometry, unless $\mathcal M$ is a parallelizable manifold (i.e., $T \mathcal M$ is diffeomorphic to the product manifold $\mathcal M \times \mathbb R^{\dim (\mathcal M)}$), $C(\mathcal M, T \mathcal M)$ does not have a basis as a $C(\mathcal M)$-module. That is, if $\mathcal M$ is not parallelizable there exists no collection $\{ W_j \}$ such that the functions $f_j$ are unique for every $V \in C(\mathcal M, T \mathcal M)$. For smooth closed manifolds it so happens that finitely many gradient fields $\nabla \phi_j$ obtained from Laplace--Beltrami eigenfunctions are a generating set for $C(\mathcal M, T \mathcal M)$, but unless $\mathcal M$ is parallelizable the frames of $\Hodge$ in \cref{lem:frame} are necessarily overcomplete.
\end{remark}

For simplicity, we shall henceforth denote the number $J'$ in \cref{lem:frame} by $\LQ$. We denote the associated index set for the frame by $\mathbb I := \mathbb N_0 \times \{1, \ldots, J\}$.

\subsection{Operator representation}

Being continuous vector fields, the frame elements $B_{i j}$ are characterized by their action as linear operators on $C^1(\mathcal M)$ functions. To evaluate this action, the SEC employs the carr\'e du champ identity~\cref{eqn:carre_du_champ} for the Laplace--Beltrami operator:
\begin{align}
    \nonumber B_{i j} \triangleright \phi_k
    &= \phi_i \grad \phi_j \triangleright \phi_k = \phi_i ( \grad \phi_j \cdot \grad \phi_k ) \\
    \label{eqn:Bij_action} &= \frac{\phi_i}{2} ( (\Delta \phi_j) \phi_k + \phi_j (\Delta \phi_k) - \Delta(\phi_j \phi_k)) = \frac{\phi_i}{2} ((\lambda_j + \lambda_k) \phi_j \phi_k - \Delta(\phi_j \phi_k)).
\end{align}
From the above we see that the action of $B_{i j}$ on the Laplace--Beltrami eigenfunctions $\phi_k$ (and, by linear extension, on linear combinations of eigenfunctions) is completely determined by the spectrum (eigenvalues) of $\Delta$ and the pointwise products of the corresponding eigenfunctions.

The operator representation of continuous vector fields also allows to compute the action of the analysis operator $T \colon \Hodge \to \ell^2(\mathbb I)$ on continuous vector fields $V \in C(\mathcal M, TM)$; that is, $TV = (V_{ij})_{(i,j) \in \mathbb I}$, where
\begin{equation}
    \label{eqn:V_anal}
    V_{ij} = \langle \phi_i, \grad \phi_j \cdot V \rangle_{L^2(\nu)} = \langle \phi_i, \grad \phi_j \cdot V \rangle_{L^2(\nu)}.
\end{equation}
Note that the second equality above reinterprets the coefficients $V_{ij}$ as the the matrix elements $\langle \phi_i, V \triangleright \phi_j \rangle_{L^2(\nu)}$ in the Laplace--Beltrami eigenfunction basis of $L^2(\nu)$.

\Cref{eqn:Bij_action,eqn:V_anal} provide a route for using any consistent spectral approximation technique for the Laplacian to perform vector field computations. Note, in particular, that these expressions do not make use of extrinsic information to the Riemannian manifold $(\mathcal M, \metric)$ such as local tangent plane approximation in the (potentially high-dimensional) embedding space $\mathbb R^d$. Moreover, \cref{eqn:V_anal} links the representation of $V$ through its coefficients under the analysis operator to its matrix representation as an operator on functions.

We now show how the coefficients $V_{i j}$ are computable using information from the pushforward vector field $\vec V \colon Y \to \mathbb R^d$, $\vec V = F_* V$, underpinning our training data.

First, since $Y \subset \mathbb R^d$ is a submanifold of Euclidean space, the action of the pushforward map $F_*$ can be expressed as linear operator application,
\begin{equation}
    \label{eqn:V_pushforward}
    \vec V(y) = F_{*x} \left(V(x)\right) = (V \triangleright F)(x), \quad y = F(x),
\end{equation}
where $V$ acts on $F$ componentwise. Let $F = \sum_{k=0}^\infty F_k \phi_k$, $F_k$ by the expansion of the embedding in the Laplace--Beltrami eigenfunction basis, where the expansion coefficients $F_k \in \mathbb R^d$ are given by the componentwise integrals $F_k = \int_{\mathcal M } \phi_k(x) F(x) \, d\nu(x)$, and the truncated sum $\sum_{k=0}^{K-1}F_k \phi_k$ converges in $C^r(\mathcal M; \mathbb R^d)$ for every $r \in \mathbb N_0$ by $C^\infty$ smoothness of $F$. Defining
\begin{displaymath}
    \hat V_{kl} = \langle (F_k \cdot \vec V) \circ F, \phi_l \rangle_{L^2(\nu)}
\end{displaymath}
and noticing that
\begin{displaymath}
    V \phi_j = V \cdot \nabla\phi_j = (F_* V \cdot F_* \nabla \phi_j) \circ F = (\vec V \circ F) \cdot (\nabla \phi_j \triangleright F),
\end{displaymath}
we get
\begin{equation}
    \label{eqn:V_coeff_hat}
    V_{ij} = \langle \phi_i, (\vec V \circ F) \cdot (\nabla \phi_j \triangleright F)\rangle_{L^2(\nu)} = \sum_{k=0}^\infty \langle (F_k \cdot \vec V) \circ F, B_{ij} \triangleright \phi_k\rangle_{L^2(\nu)} = \sum_{k,l=0}^\infty \hat V_{kl}\langle \phi_l, B_{ij} \triangleright \phi_k\rangle_{L^2(\nu)}.
\end{equation}

Next, let $\{\varepsilon_{i j}\}_{(i,j) \in \mathbb I}$ be the standard orthonormal basis of $\ell^2(\mathbb I)$ for an enumeration $\varepsilon_{i_1,j_1}, \varepsilon_{i_2, j_2}, \ldots$ of the indices $(i, j) \in \mathbb I = \mathbb N_0 \times \{1, \ldots, \LQ \}$, and similarly let $\{\hat\varepsilon_{ij}\}_{(i,j)\in \hat{\mathbb I}}$ be the standard orthonormal basis of $\ell^2(\hat{\mathbb I})$ for the index set $\hat{\mathbb I} = \mathbb N_0^2$. Define the linear map $\mathcal F \colon L^2(\nu_*; \mathbb R^d) \to \ell^2(\hat{\mathbb I})$ as
\begin{displaymath}
    \mathcal F \vec W = \sum_{(k,l)\in \hat{\mathbb I}} \langle (F_k \cdot \vec W)\circ F, \phi_l \rangle_{L^2(\nu)} \hat \varepsilon_{kl}.
\end{displaymath}
Define also $\mathcal D \colon \dom(\mathcal D) \to \ell^2(\mathbb I)$ on a dense domain $\dom(\mathcal D) \subset \ell^2(\hat{\mathbb I})$ through its matrix representation,
\begin{displaymath}
    \langle \varepsilon_{ij}, \mathcal D \hat \varepsilon_{kl}\rangle_{\ell^2(\mathbb I)} = \langle \phi_l, B_{ij}\triangleright \phi_k\rangle_{L^2(\nu)}.
\end{displaymath}
We will verify in \cref{sec:convanalysis} that $\mathcal F$ is bounded and its range is a subspace of $\dom(\mathcal D)$.

With these definitions, \cref{eqn:V_coeff_hat} becomes
\begin{displaymath}
    V_{ij} = (\mathcal D \mathcal F \vec V)_{ij}, \quad \forall (i,j) \in \mathbb I.
\end{displaymath}
Inserting the last expression in~\cref{eqn:V_anal}, we deduce that every sequence $b \in \ell^2(\mathbb I)$ that reconstructs $V$ in the frame (i.e., $V = T^* b$) satisfies
\begin{equation}
    \label{eqn:b_regression}
    G b = \mathcal D \mathcal F \vec V.
\end{equation}

\Cref{eqn:b_regression} provides the theoretical basis for setting up the vector field learning problem as a problem in linear regression. As we will see below, the Gram operator $G$ in the left-hand side is computed using the SEC framework, whereas the vector of coefficients $\mathcal D \mathcal F \vec V$ in the right-hand side can be empirically approximated from the vector field samples $\vec v_n \in \mathbb R^d$.

\subsection{Building the frame elements} \label{subsec:frame}

We now describe various SEC constructions that will be used in our vector field approximation scheme. These constructions involve scalar coefficients derived from the spectral data $(\lambda_j, \phi_j)$ of the Laplacian, which we name $c$-, $g$-, and $d$- coefficients, respectively.

We first consider the $c$-coefficients,
\begin{equation} \label{eqn:def:c_coeff}
    c_{i j p} := \langle \phi_p, \phi_i \phi_j \rangle_{L^2(\nu)}, \quad \forall i,j,p \in \num_0.
\end{equation}
These coefficients provide a representation of products of Laplace--Beltrami eigenfunctions,
\[ \phi_i \phi_j = \sum_{p=0}^{\infty} c_{i j p} \phi_p , \quad \forall i,j,p \in \num_0,\]
where the summation over $p$ giving $\phi_i(x)\phi_j(x)$ converges uniformly over $x \in \mathcal M$ by compactness of $\mathcal M$ and $C^\infty$ smoothness of the $\phi_j$. 

Next, we consider the $g$-coefficients,
\begin{equation} \label{eqn:def:g_coeff}
    g_{p j k} := \left\langle \phi_p, \grad \phi_j \cdot \grad \phi_k \right\rangle_{L^2(\nu)}, \quad \forall p, j, k \in \num_0.
\end{equation}
These coefficients enable a representation of Riemannian inner products of gradient vector fields,
\begin{displaymath}
    \grad \phi_j \cdot \grad \phi_k = \sum_{p=0}^{\infty} g_{p j k} \phi_p,
\end{displaymath}
where the sums over $p$ converge again uniformly on $\mathcal M$. Using~\cref{eqn:Bij_action}, we get
\begin{equation}
    \label{eqn:g_coeff_Lap}
    g_{p j k} = \frac{1}{2} \paran{\lambda_j + \lambda_k - \lambda_p} c_{j k p}.
\end{equation}

Finally, the $d$-coefficients are defined as
\begin{equation}
    \label{eqn:d_coeff_frame}
    d_{i j k l} = \langle \phi_l, B_{i j} \triangleright \phi_k \rangle_{L^2(\nu)}, \quad \forall i,k,l \in \mathbb N_0, \quad \forall j \in \{ 1, \ldots, \LQ \}.
\end{equation}
These coefficients represent the action $B_{i j} \triangleright \phi_k$ of the frame elements on the basis functions, or, equivalently, the matrix elements of the $\mathcal D$ operator from \cref{subsec:V_op}, $\langle \varepsilon_{ij}, \mathcal D \hat\varepsilon_{kl}\rangle_{\ell^2(\hat{\mathbb I})} = d_{ijkl}$. Using again~\cref{eqn:Bij_action}, we obtain
\begin{equation}
    \label{eqn:def:d_coeff}
    d_{i j k l} = \langle \phi_l, B_{i j} \triangleright \phi_k \rangle_{L^2(\nu)} = \langle \phi_i \phi_{l}, \grad \phi_j \cdot \grad \phi_k \rangle_{L^2(\nu)} = \sum_{p=0}^{\infty} c_{i l p} g_{p j k}.
\end{equation}
A restriction of the $d$-coefficients lead to the matrix elements of the Gram operator $G \colon \ell^2(\mathbb I) \to \ell^2(\mathbb I)$,
\begin{equation}
    \label{eqn:Gramian}
    G_{i j k l} := \langle \varepsilon_{i j}, G \varepsilon_{k l}\rangle_{\ell^2} = \langle T^* \varepsilon_{i j}, T^* \varepsilon_{k l}\rangle_{\Hodge} = d_{ijlk}, \quad (i,j), (k,l) \in \mathbb{I}.
\end{equation}
Using the $d$-coefficients, we can also represent the pushforward $\vec V = F_*V$ of $V$ in data space from~\cref{eqn:V_pushforward} :
\begin{equation}
    \label{eq:FV}
    \vec V = \sum_{i,k,l=0}^\infty \sum_{j=1}^{\LQ } b_{i j} F_k \langle \phi_l, B_{i j} \triangleright \phi_k \rangle_{L^2(\nu)} \varphi_l = \sum_{i,k,l=0}^\infty \sum_{j=1}^{\LQ } d_{i j k l} b_{i j} F_k \varphi_l,
\end{equation}
where the summation over $i,j,k,l$ converges in $L^2(\nu_*; \mathbb R^d)$ norm.

\Cref{eq:FV} establishes the pushforward of the vector field as an infinite sum of various coefficients. Among these, the $d$-coefficients depend only on the geometry of $\calM$, the $F$-coefficients depend only on the embedding, and the $b$-coefficients are the only coefficients that depend on the vector field $V$. We next describe how the various infinite sums determining these coefficients may be truncated to finite sums. This is essential for our eventual data-driven realization of \cref{eq:FV}. 

\subsection{The truncated problem and vector field projection}
\label{subsec:projection}

We fix a positive integer $L\in\num$ which we call the \emph{spectral resolution parameter}. It determines the hypothesis space $\Hodge_L \subset \Hodge$ of dimension $D_L = L \LQ $ which we shall use for the approximation of the vector field $V$:
\begin{displaymath}
    \Hodge_L = \spn \{ B_{i j}: \text{$0 \leq i \leq {L-1}$, $1 \leq j \leq \LQ $} \}.
\end{displaymath}
Let similarly $\ell^2(\mathbb I_L) = \spn \{ \varepsilon_{i j}: \text{$0 \leq i \leq L$, $1 \leq j \leq \LQ $} \} \subset \ell^2(\mathbb I)$, $\dim \ell^2(\mathbb I_L) = D_L$, and define $\pi_L \colon \ell^2 \to \ell^2$ as the orthogonal projection with $\ran\pi_L = \ell^2(\mathbb I_L)$.

Associated with $\pi_L$ is a projected analysis operator $T_L := \pi_L T$, whose adjoint $T_L^* = T^* \pi_L$ also maps into $\Hodge_L$. Moreover, since $\pi_L$ is an orthogonal projection, we have $\pi_L^+ = \pi_L$ and the pseudoinverse of $T_L$ is given by $T_L^+ = \pi_L T^+$. Defining the projected Gram operator
\begin{displaymath}
    G_L := \pi_L G \pi_L = T_L T_L^*,
\end{displaymath}
we have $G_L^+ = \pi_L G^+ \pi_L$. Moreover, $G_L$ is represented by a $D_L \times D_L$ symmetric matrix $\bm G_L$ with elements
\begin{displaymath}
    G_{mn} = \langle \varepsilon_{i_m j_m}, G \varepsilon_{i_n j_n}\rangle_{\ell^2(\mathbb I)} = \langle B_{i_m, j_m}, B_{i_n,j_n}\rangle_{\Hodge} = d_{i_m j_m j_n i_n},
\end{displaymath}
whose pseudoinverse $\bm G_L^+ = [G^+_{mn}]$ gives the matrix elements $ \langle \varepsilon_{i_m,j_m}, G^+_L \varepsilon_{i_n, j_n}\rangle_{\ell^2(\mathbb I)} = G^+_{mn} $ of $G_L^+$.

Let $\mathcal D_L = \pi_L \mathcal D$. We define a spectrally truncated approximation $V^{(L)} \in C(\mathcal M, T \mathcal M)$ of the dynamical vector field $V$ as
\begin{equation}
    \label{eqn:minim}
    V^{(L)} := T^* b^{(L)}, \quad b^{(L)} = G_L^+ \mathcal D_L \mathcal F V,
\end{equation}
where $b^{(L)} \in \mathbb \ell^2(\mathbb I_L)$ is a minimum-norm solution of the regression problem (cf.~\cref{eqn:b_regression})
\begin{equation}
    \label{eqn:b_regression_trunc}
    G_L b^{(L)} = \mathcal D_L \mathcal F V.
\end{equation}
Equivalently, $V^{(L)}$ is the best-fit approximation of $V$ among vector fields in the hypothesis space $\mathbb H_L$, in the sense of minimizing the square error functional $\mathcal E_L \colon \Hodge_L \to \mathbb R $, $\mathcal E_L(W) = \lVert  V - W\rVert_{\Hodge_L}^2$.

As $L \to \infty$, $\pi_L$ converges strongly to the identity on $\ell^2(\mathbb I)$, and thus $T_L$ and $G_L^+$ converge to $T$ and $G$, respectively. We deduce that $b^{(L)}$ converges to a (minimum-norm) solution $b$ of~\cref{eqn:b_regression} and $V^{(L)}$ converges to $V = T^* b$ in the Hodge norm,
\begin{equation} \label{eqn:VL_to_V}
    \lim_{L\to\infty}\left\lVert V^{(L)} - V \right\rVert_{\mathbb H} = 0.
\end{equation}

Next, we can write down an expansion for $V^{(L)}$ in the frame,
\begin{displaymath}
    V^{(L)} = \sum_{i=0}^{L-1} \sum_{j=1}^{\LQ } b_{i j}^{(L)} B_{i j}, \quad b_{ij}^{(L)} = \langle \varepsilon_{ij}, b^{(L)}\rangle_{\ell^2(\mathbb I)}.
\end{displaymath}
This expansion leads to an approximation $\vec V^{(L)} = F_* V^{(L)}$ of the pushforward vector field $\vec V$, given by (cf.\ \cref{eq:FV})
\begin{equation}
    \label{eqn:FVL}
    \vec V^{(L)} = \sum_{i=0}^{L-1} \sum_{j=1}^J \sum_{k,l=0}^\infty d_{ijkl} b_{ij}^{(L)} F_k \varphi_l.
\end{equation}

Note that~\eqref{eqn:FVL} gives $V^{(L)}$ as a finite linear combination of infinite-rank (unbounded) operators $B_{ij}$. Similarly, the regression problem yielding the coefficients $b^{(L)}$ involves an infinite-rank (bounded) operator, $\mathcal F$, in the right-hand side. To arrive at an approximation that is more amenable to numerical implementation, we make a number of modifications to~\cref{eqn:FVL}. In what follows, $H_M = \spn \{ \phi_0, \ldots, \phi_{M-1} \}$ will be the $M$-dimensional subspace of $L^2(\nu)$ spanned by the leading $M$ Laplace--Beltrami eigenfunctions, and $\Pi_M \colon L^2(\nu) \to L^2(\nu)$ the orthogonal projection mapping into this subspace. Given $M_1, M_2 \in \mathbb N$, we let $\ell^2(\hat{\mathbb I}_{M_1, M_2})$ be the $(M_1M_2)$-dimensional subspace of $\ell^2(\hat{\mathbb I})$ corresponding to the index set $\hat{\mathbb I}_{M_1, M_2} = \{ (i, j) \in \mathbb N_0 \times \mathbb N_0: i < M_1, j < M_2 \}$, and define $\hat\pi_{M_1, M_2} \colon \ell^2(\hat{\mathbb I}) \to \ell^2(\hat{\mathbb I})$ as the orthogonal projection mapping onto this subspace.
\begin{itemize}
    \item We approximate the $d$-coefficients by truncating the summation in~\cref{eqn:def:d_coeff} to $\LD \in \mathbb N$ terms,
        \begin{displaymath}
            d^{(\LD)}_{ijkl} := \sum_{p=0}^{\LD - 1} c_{ilp} g_{pjk}.
        \end{displaymath}
        This truncation induces densely-defined operators $B_{ij}^{(\LD)}$ on $L^2(\nu)$ defined via the matrix representation
        \begin{displaymath}
            \langle \phi_l, B_{ij}^{(\LD)} \phi_k\rangle_{L^2(\nu)} = d_{ijkl}^{(\LD)}.
        \end{displaymath}
        We view $B_{ij}^{(\LD)}$ as approximations of the vector fields $B_{ij}$.
    \item Using the $B_{ij}^{{(\LD)}}$, we introduce $ G^{(L,\LD)} \colon \ell^2(\mathbb I) \to \ell^2(\mathbb I)$ as an approximation of the Gram operator $G_L$ with matrix elements (cf.\ \cref{eqn:Gramian})
        \begin{displaymath}
            G^{(L,\LD)}_{i j k l} := \langle \varepsilon_{i j}, G^{(L,\LD)} \varepsilon_{k l}\rangle_{\ell^2(\mathbb I)} := d^{(\LD)}_{ijlk}.
        \end{displaymath}
        We also define $\mathcal D^{(L,\LD)} \colon \dom(\mathcal D^{(L,\LD)}) \to \ell^2(\mathbb I)$ as an approximation of $\mathcal D_L$ on a dense domain $\dom(\mathcal D^{(L,\LD)}) \subset \ell^2(\hat I)$, where
        \begin{displaymath}
            \langle \varepsilon_{ij}, \mathcal D^{(L,\LD)} \hat \varepsilon_{kl}\rangle_{\ell^2(\mathbb I)} = d_{ijkl}^{(\LD)}.
        \end{displaymath}
    \item Let $G^{(L,\LD)} u_j = \gamma_j u_j$ be an eigendecomposition of $G^{(L,\LD)}$ where the eigenvalues $\gamma_j$ are non-negative and the eigenvectors $u_j$ are orthonormal in $\ell^2(\mathbb I)$. For $\eta>0$, we approximate $G^{(L,\LD)}$ by the spectrally truncated Gram operator $G^{(L,\eta,\LD)} \colon \ell^2(\mathbb I) \to \ell^2(\mathbb I)$, where
        \begin{equation}
            \label{eqn:spectrally_truncated_gram}
            G^{(L,\eta,\LD)} = \sum_{j: \gamma_j>\eta} \gamma_j u_j \langle u_j, \cdot \rangle_{\ell^2(\mathbb I)}.
        \end{equation}
        Note that the pseudoinverse of this operator is given by
        \begin{displaymath}
            G^{(L,\eta,\LD),+} = \sum_{j: \gamma_j>\eta} \gamma_j^{-1} u_j \langle u_j, \cdot \rangle_{\ell^2(\mathbb I)},
        \end{displaymath}
        and has norm at most $1/\eta$ by construction.
    \item With $L_1, L_2 \in \mathbb N$, we replace the regression problem for $b^{(L)}$ in~\cref{eqn:b_regression_trunc}  by
        \begin{equation}
            \label{eqn:b_regression_trunc_2}
            G^{(L,\eta, \LD)} b^{(L,L_1,L_2,\eta,\LD)} = \mathcal D^{(L,\LD)} \mathcal F^{(L_1,L_2)} \vec V,
        \end{equation}
        where $\mathcal F^{(L_1, L_2)} = \hat \pi_{L_1, L_2} \mathcal F$. We use the solution
        \begin{equation}
            \label{eqn:minim_2}
            b^{(L,L_1, L_2,\eta,\LD)} = G^{(L,\eta,\LD)+} \mathcal D^{(L,\LD)} \mathcal F^{(L_1,L_2)} \vec V \in \ell^2(\mathbb I_L)
        \end{equation}
        of this problem as an approximation of $b^{(L)}$ from~\cref{eqn:minim}.
    \item We approximate $B_{ij}^{(\LD)}$ by finite-rank operators $B_{ij}^{(L_1, L_2, \LD)} \colon L^2(\nu) \to L^2(\nu)$ where
        \begin{displaymath}
            B_{ij}^{(L_1, L_2, \LD)} = \Pi_{L_2} B_{ij}^{(\LD)} \Pi_{L_1}.
        \end{displaymath}
\end{itemize}
 Letting $ \ell = (L, L_1, L_2, \eta, \LD)$ collectively represent the truncation parameters introduced in this subsection, we arrive at the following modified version of~\cref{eqn:FVL},
\begin{equation}
    \label{eqn:VL2}
    \vec V^{(\ell)} :=  \sum_{i=0}^{L-1} \sum_{j=1}^J \sum_{k=0}^{L_1-1} \sum_{l=0}^{L_2-1} d_{ijkl}^{(\LD)} b_{ij}^{(\ell)} F_k \varphi_l,
\end{equation}
where $b_{ij}^{(\ell)} = \langle \varepsilon_{ij}, b^{(\ell)}\rangle_{\ell^2(\mathbb I)}$.

The vector field $\vec V^{(\ell)}$ in~\cref{eqn:VL2} can be employed as a practical approximation to $\vec V$ assuming that (finitely many) coefficients $d_{ijkl}^{(\LD)}, b_{ij}^{(L,\LD)}, F_k$ and basis functions $\varphi_l$ are known. In \cref{sec:data_driven} we will relax that assumption and take up the problem of data-driven approximation of these objects using kernel methods.

\begin{remark}
    As we will discuss in \cref{subsec:convergence_vector_field}, the spectral truncation of the Gram operators by $\eta$ from below is introduced in order to ensure convergence of the pseudoinverse of $G^{(L,\eta,\LD)}$ as $\LD \to \infty$. In particular, the approximation of $G_L$ by $G^{(L,\LD)}$ is not an approximation by projection (in contrast to the approximation of $G$ by $G_L$), and we cannot deduce convergence of $G^{(L,\LD),+}$ as $\LD \to \infty$.
\end{remark}

\section{Data-driven formulation}
\label{sec:data_driven}

The data-driven formulation of our scheme is built on kernel-based approximations of the Laplace--Beltrami eigenpairs, $(\lambda_j, \phi_j)$, as well as the corresponding extensions $\varphi_j$ of the eigenfunctions on data space. With such approximations, the data-driven scheme is essentially structurally identical to the approach presented in \cref{subsec:projection} leading to the projected vector field $\vec V^{(\ell)}$ in~\cref{eqn:VL2}.

We begin in \cref{subsec:kernel} by stating abstract conditions on kernel methods that are sufficient to guarantee asymptotic consistency of our scheme in the limit of large data. A concrete example satisfying these conditions is the diffusion maps algorithm \cite{CoifmanLafon2006}, which is used in the numerical experiments of \cref{sec:numexp} and summarized in \cref{app:dm}. In \cref{subsec:secalgo}, we present the associated SEC-based scheme for learning vector fields.

In what follows, $L^2(\mu_N)$ will be the $N$-dimensional $L^2$ space on $\mathcal M$ associated with the sampling measure $\mu_N$. Note that elements of $L^2(\mu_N)$ are equivalence classes of functions $f$ with values $f(x_n)$ known at the sampling points $x_n \in \mathcal M$. As such, every such element can be represented by a column vector $\bm f = (f(x_1), \ldots, f(x_N))^\top \in \mathbb R^N$ and every linear operator $A \colon L^2(\mu_N) \to L^2(\mu_N)$ is representable by an $N \times N$ matrix $\bm A$ such that $\bm A \bm f$ is the column vector representation of $A f$.

\subsection{Kernel-based approximation of the Laplacian}
\label{subsec:kernel}

The approximations of the Laplacian used in this paper are based on normalized kernels $\rho_{\epsilon,N} \colon \mathbb R^d \times \mathbb R^d \to \mathbb R_{>0}$, parameterized by a bandwidth (lengthscale) parameter $\epsilon>0$. The kernels are $C^\infty$ smooth, and are constructed using the training dataset $\{y_n\}_{n=0}^{N-1} \subset \mathbb R^d$ so that they are Markovian with respect to the sampling measure $\mu_N$; that is,
\begin{equation}
    \label{eqn:markov}
    \int_{\mathcal M} \rho_{\epsilon,N}(y, F(x)) \, d\mu_N(x) = 1, \quad \forall y \in \mathbb R^d.
\end{equation}

Every such kernel pulls back to a $C^\infty$ kernel $p_{\epsilon,N} \colon \mathcal M \times \mathcal M \to \mathbb R_{>0}$ on the manifold, $p_{\epsilon,N}(x, x') = \rho_{\epsilon,N}(F(x), F(x'))$, with an associated integral operator $P_{\epsilon,N} \colon L^2(\mu_N) \to L^2(\mu_N)$,
\begin{displaymath}
    P_{\epsilon,N} f = \int_{\mathcal M} p_{\epsilon,N}(\cdot, x) f(x) \, d\mu_N(x) \equiv \frac{1}{N}\sum_{n=1}^N \rho_{\epsilon,N}(F(\cdot), x_n) f(x_n).
\end{displaymath}
Defining
\begin{equation}
    \label{eqn:lapl_epsilon}
    \Delta_{\epsilon,N} := - \frac{\log P_{\epsilon,N}}{c \epsilon^2}
\end{equation}
for an $\epsilon$-independent constant $c$ that depends on the parametric family of the kernel $\rho_{\epsilon,N}$ (but not on $\epsilon$ or $N$), our data-driven approach employs kernels such that $\Delta_{\epsilon,N}$ is an approximation of the Laplacian $\Delta$ and $P_{\epsilon,N}$ is an approximation of the heat operator $e^{-c \epsilon^2 \Delta}$. In addition, $\rho_{\epsilon,N}$ is constructed such that $P_{\epsilon,N}$ is ergodic (i.e., its eigenvalue equal to 1 is simple) and reversible (i.e., it satisfies detailed balance),
\begin{equation}
    \label{eqn:detailed_balance}
    w_{\epsilon,N}(x)p_{\epsilon,N}(x, x') = w_{\epsilon,N}(x') p_{\epsilon,N}(x', x), \quad \text{$\mu_N$-a.e., $x, x' \in \mathcal M$}.
\end{equation}

Recall the sampling density $\sigma$ introduced in \Cref{A:3}. In~\eqref{eqn:detailed_balance}, $w_{\epsilon,N} \in L^1(\mu_N)$ is a vector with strictly positive elements that provides an estimate of the reciprocal $\frac{1}{\sigma}$, i.e., the density of the volume measure relative to the sampling measure of the training data. To build this estimate, we normalize $w_{\epsilon,N}$ such that $\int_{\mathcal M}w_{\epsilon,N}\, d\mu_N$ is an approximation of the volume $\nu(\mathcal M)$ of the manifold (obtained using heat kernel asymptotics; see \cref{app:dm}). With this normalization, empirical integrals $\int_{\mathcal M} f \, d\nu_{\epsilon,N}$ of continuous functions $f \in C(\mathcal M)$ with respect to the weighted measure
\begin{equation}
    \label{eqn:weighted_measure}
    \nu_{\epsilon,N} := \sum_{n=1}^N w_{\epsilon,N}(x_n)\delta_{x_n},
\end{equation}
approximate integrals $\int_{\mathcal M} f \, d\nu$ with respect to the volume measure $\nu$ of the manifold. Reversibility also implies that $P_{\epsilon,N}$ is diagonalizable and has real eigenvalues. When convenient, we will use the symbol $\alpha$ to collectively denote data-driven approximation parameters; e.g., $\alpha = (\epsilon, N)$ in the present context.

Next, let us consider an eigendecomposition
\begin{equation}
    \label{eqn:P_eigs}
    P_\alpha \phi_j^{(\alpha)} = \Lambda_j^{(\alpha)} \phi_j^{(\alpha)},
\end{equation}
where the eigenvalues $\Lambda_j^{(\alpha)} \in \mathbb R$ are bounded in modulus by 1 since $P_\alpha$ is a Markov operator and the eigenvectors $\phi_j^{(\alpha)}$ are normalized to unit norm on $L^2(\nu_\alpha)$. Moreover, $w_\alpha$ is an eigenvector of $P^*_\alpha$ at eigenvalue 1, $P^*_\alpha w_\alpha = w_\alpha$. Observe that~\cref{eqn:P_eigs} is equivalent to the $N \times N$ matrix eigenvalue problem
\begin{displaymath}
    \bm P_\alpha \bm \phi_j^{(\alpha)} = \Lambda_j^{(\alpha)} \bm \phi_j^{(\alpha)},
\end{displaymath}
where $\bm P_\alpha = [P_{ij}^{(\alpha)}]_{i,j=1}^N$ is a stochastic matrix with entries $P_{ij}^{(\alpha)} = p_\alpha(x_i, x_j) / N$ and the eigenvectors $\bm \phi_j^{(\alpha)} \in \mathbb R^N$ are normalized such that $\bm \phi^{(\alpha), \top}_j \bm W_\alpha \bm \phi^{(\alpha)}_j = 1$ for the diagonal matrix $\bm W_\alpha = \diag(w_\alpha(x_1), \ldots, w_\alpha(x_N))$. Note also that every eigenvector $\phi_j^{(\alpha)}$ with nonzero corresponding eigenvalue $\Lambda_j^{(\alpha)}$ has a smooth extension $\varphi_j^{(\alpha)} \in C^\infty(\mathbb R^d)$ on data space,
\begin{equation}
\label{eqn:def:eigenext}
    \varphi_j^{(\alpha)}(y) = \frac{1}{\Lambda_j^{(\alpha)}} \int_{\mathcal M} \rho_\alpha(y, x) \phi_j(x) \, d\mu_N(x),
\end{equation}
satisfying $\varphi_j^{(\alpha)}(y_n) = \phi_j^{(\alpha)}(x_n)$ for every $n \in \{ 1, \ldots, N \}$.

The operator $\Delta_\alpha$ has an eigendecomposition
\begin{displaymath}
    \Delta_\alpha \phi_j^{(\alpha)} = \lambda_j^{(\alpha)} \phi_j^{(\alpha)}
\end{displaymath}
with eigenvalues $\lambda_j^{(\alpha)} = - \frac{1}{c\epsilon^2} \log\Lambda_j^{(\alpha)}$. For simplicity of exposition, we will assume that all eigenvalues $\Lambda^{(\alpha)}_j$ are nonnegative, and can thus be ordered as $1 = \Lambda^{(\alpha)}_0 \geq \Lambda^{(\alpha)}_1 \geq \cdots \geq \Lambda^{(\alpha)}_{N-1} \geq 0$ by Markovianity of $P_\alpha$.

As we will see in \cref{subsec:secalgo} below, the following assumption encapsulates sufficient conditions for convergence of our data-driven scheme.

\begin{Assumption} \label{A:4}
    The operators $\Delta_\alpha$ converge spectrally to the Laplacian $\Delta$ as $\alpha = (\epsilon, N)$ tends to $(0^+, \infty)$ along a sequence $\alpha_N = (\epsilon_N, N)$, in the following sense:
    \begin{enumerate}[(i)]
        \item For every $j \in \mathbb N$, the eigenvalues $\lambda^{(\alpha_N)}_j$ converge to the Laplace--Beltrami eigenvalue $\lambda_j$, including multiplicities.
        \item For every Laplace--Beltrami eigenfunction $\phi_j \in C^\infty(\mathcal M)$ corresponding to $\lambda_j$ there exists a sequence of eigenvectors $\phi_j^{(\alpha_N)} \in L^2(\mu_N)$ of $\Delta_\alpha$ corresponding to $\lambda_j^{(\alpha)}$  whose representatives $\varphi_j^{(\alpha_N)}$ converge uniformly to $\phi_j$ on $\mathcal M$:
            \begin{displaymath}
                \lim_{N\to\infty} \left\lVert \varphi_j^{(\alpha_N)} \circ F - \phi_j \right\rVert_{C(\mathcal M)} = 0.
            \end{displaymath}
        \item The convergence of the eigenvalues in (i) takes place at a sufficiently rapid rate such that
            \begin{displaymath}
                \lim_{N\to\infty} \epsilon_N^{\dim(\mathcal M)} \sum_{j=0}^{N-1} \left(e^{-\epsilon_N^2 \lambda_j^{(\alpha_N)}} - e^{-\epsilon^2_N \lambda_j}\right) = 0.
            \end{displaymath}
    \end{enumerate}
\end{Assumption}

A number of techniques from the field of manifold learning satisfy \cref{A:4}; e.g., \cite{TrillosEtAl_error_2020, TaoShi2020lap,WormellReich21,CalderTrillos22,PeoplesHarlim23}. Convergence results from these methods are summarized in \cref{tab:rates}. As mentioned earlier, the numerical examples presented in this paper utilize the diffusion maps algorithm \cite{CoifmanLafon2006}, which is a popular kernel method for approximation of Laplacians and weighted Laplacians under potentially nonuniform sampling distributions $\mu$ with respect to the volume measure $\nu$ of the manifold. A summary of diffusion maps as used in this paper is included in \cref{app:dm}. The paper \cite{WormellReich21} proves spectral convergence of this techniques for data sampled on flat tori. This analysis thus covers the numerical examples in \cref{subsec:circlerotation}, which take place on circles. In \cref{subsec:torusroatation}, we consider examples on tori with curvature, which lie outside the theoretical scope of \cite{WormellReich21}. Nevertheless, there is strong numerical evidence that diffusion maps provides approximations of the Laplacian of sufficient spectral accuracy even in those cases.

\begin{remark}
    In supervised and unsupervised learning applications utilizing the $\phi_j^{(\alpha)}$ as basis vectors (e.g., feature extraction \cite{CoifmanLafon2006} and kernel regression \cite{AlexanderGiannakis20}) the normalization of the weight vector $w_\alpha$ can oftentimes be chosen arbitrarily without affecting the results. On the other hand, the SEC approach presented in this paper uses the product structure of the $\phi_j^{(\alpha)}$ in conjunction with the eigenvalues of $\Delta_\alpha$ to implement the carr\'e du champ identity of the Laplacian \eqref{eqn:carre_du_champ} (through data-driven analogs of the $c$- and $g$-coefficients; see \cref{subsec:secalgo} below). Consistency of this implementation requires that the $\phi_j^{(\alpha)}$ are normalized consistently with the volume form of the manifold. Geometrically, this stems from the fact that the volume $\nu(\mathcal M)$ is related to the trace of the heat operator (i.e., the sum of its eigenvalues). \Cref{A:4}(iii) is essentially a trace condition on $P_\alpha$ that ensures that $\nu(\mathcal M)$ can be consistently approximated using the trace of the matrices $\bm P_\alpha$; see \cref{app:dm}.
\end{remark}

\begin{table}
    \caption{Representative data-driven approximation techniques with spectral convergence for the Laplacian. \Cref{A:4} addresses spectral convergence as the data parameters $\alpha := (\epsilon,N)$ approach $(0^+, \infty)$. Convergence results have been established for various approaches to this limit, and under various assumptions on the manifold $\calM$. The table lists a few such results in chronological order of the corresponding reference in the literature. All paths of approach of $\alpha$ to $(0^+, \infty)$ are along a sequence $\alpha_N = (\epsilon_N, N)$. The variable $m$ represents the dimension of the manifold $\calM$, $p_m$ is a constant determined by $m$, and $k$ is the maximum multiplicity among a fixed number of top eigenvalues. }
    \small
    \centering
    \begin{tabular}{lll}
        Reference & Assumptions on the manifold & Rate law $\epsilon_N$ \\
        \hline
        \cite{TrillosEtAl_error_2020} & Closed & $O\paran{ \log(N)^{p_m} N^{1/2m} }$ \\
        \cite{TaoShi2020lap} & Closed & $\epsilon^{1/2}_N + \frac{\ln N - \ln \epsilon_N + 1}{ \epsilon_N^{k+3} \sqrt{N} } = o(1)$\\
        \cite{WormellReich21} & Torus & N/A \\
        \cite{CalderTrillos22} & Closed & $O \paran{ \paran{ \ln N / N }^{1/O(m)} }$ \\
        \cite{PeoplesHarlim23} & Compact with boundary &  $O \paran{ \paran{ \ln N / N }^{1/O(m)} }$ \\
        \hline
    \end{tabular}
    \label{tab:rates}
\end{table}

\subsection{Data-driven SEC framework}
\label{subsec:secalgo}

Let $\ell = (L,L_1, L_2,\eta, \LD)$ be the vector of discretization parameters from \cref{subsec:projection} and set $M = \max \{ L, L_1, L_2, \LD \}$. By \cref{A:4}, for a sufficiently large number of samples $N$, the leading $M$ eigenvectors $\phi^{(\alpha)}_0, \ldots, \phi^{(\alpha)}_{M-1}$ are linearly independent and form a basis of an $M$-dimensional subspace $H_{M,N} \subseteq L^2(\nu_\alpha)$ for the empirical volume measure $\nu_\alpha$ from~\eqref{eqn:weighted_measure}. By convention, we always choose the basis vectors so that they are orthonormal in this space, $\langle \phi_i^{(\alpha)}, \phi_j^{(\alpha)}\rangle_{L^2(\nu_\alpha)} = \delta_{ij}$. We also define $\Pi_{M,N} \colon L^2(\nu_\alpha) \to L^2(\nu_\alpha)$ as the orthogonal projection with $\ran \Pi_{M,N} = H_{M,N}$. Similarly to the expansion of the embedding $F \colon \mathcal M \to \mathbb R^d$ in the Laplace-Beltrami eigenfunction basis, we can write down an expansion $F^{(\alpha)} = \sum_{k=0}^{N-1} F_k^{(\alpha)} \phi_k^{(\alpha)} \in L^2(\nu_\alpha)$ with coefficients $F_k^{(\alpha)} \in \mathbb R^d$ given by componentwise discrete integrals
\begin{displaymath}
    F_k^{(\alpha)} = \int_{\mathcal M} \phi^{(\alpha)}_k F\, d\nu_\alpha = \sum_{n=1}^N \phi^{(\alpha)}_k(x_n) y_n w_\alpha(x_n).
\end{displaymath}
This expansion recovers $F$ on the training dataset (i.e., $\mu_N$-a.e.), $F^{(\alpha)}(x_n) = F(x_n)$ for all $n \in \{ 1, \ldots, N \}$.

The data-driven implementation of our framework is based on the following approximations of the $c$-, $g$-, and $d$-coefficients from \cref{sec:dynamicslearn}:
\begin{equation}
     \label{eqn:approx:d_coeff}
     \begin{gathered}
         c_{ijp}^{(\alpha)} := \left\langle \phi_p^{(\alpha)}, \phi_i^{(\alpha)} \phi_j^{(\alpha)}\right\rangle_{L^2(\nu_\alpha)}, \quad
         g_{pjk}^{(\alpha)} := \frac{1}{2} \left(\lambda^{(\alpha)}_j + \lambda^{(\alpha)}_k - \lambda^{(\alpha)}_p\right) c_{jkp}^{(\alpha)}, \\
         d_{ijkl}^{(\LD,\alpha)} := \sum_{p=0}^{\LD -1} c_{ilp}^{(\alpha)} g_{pjk}^{(\alpha)}.
     \end{gathered}
\end{equation}

With these approximations, we define operators $G^{(\LD,\alpha)} \colon \ell^2(\mathbb I_N)\to \ell^2(\mathbb I_N)$ and $\mathcal D^{(\LD,\alpha)} \colon \ell^2(\mathbb{\hat I_{N,N}}) \to \ell^2(\mathbb I_{N,N})$ of the $G^{(\LD)}$ and $\mathcal D^{(\LD)}$ operators from \cref{subsec:projection} to have matrix elements
\begin{equation}
    \label{eqn:g_d_ops_data_driven}
    \langle \varepsilon_{i j}, G^{(\LD,\alpha)} \varepsilon_{k l}\rangle_{\ell^2(\mathbb I)} = d^{(\LD,\alpha)}_{ijlk}, \quad \langle \varepsilon_{ij}, \mathcal D^{(\LD,\alpha)} \hat \varepsilon_{kl}\rangle_{\ell^2(\mathbb I)} = d_{ijkl}^{(\LD,\alpha)}.
\end{equation}
For $L \in \mathbb N$ and $\eta>0$, we also define projected Gram operators $G^{(L, \LD,\alpha)} = \pi_L G^{(\LD,\alpha)}\pi_L$ and their spectrally truncated approximations $G^{(L,\eta,\LD, \alpha)}$ on $\ell^2(\mathbb I_N)$ analogously to $G^{(L,\LD)}$ and $G^{(L,\eta,\LD)}$ from \cref{subsec:projection}, respectively.

Next, as a data-driven approximation of $\mathcal F \colon L^2(\nu_*; \mathbb R^d) \to \ell^2(\hat{\mathbb I})$, we introduce the operator $\mathcal F^{(\alpha)} \colon L^2(\nu_{\alpha*}; \mathbb R^d) \to \ell^2(\hat{\mathbb I}_{N,N})$, where
\begin{displaymath}
    \mathcal F^{(\alpha)} \vec W = \sum_{(k,l)\in \hat{\mathbb I}_{N,N}} \langle (F_k^{(\alpha)} \cdot \vec W)\circ F, \phi_l^{(\alpha)} \rangle_{L^2(\nu_\alpha)} \hat \varepsilon_{kl}.
\end{displaymath}
Note that the action of this map on the pushforward vector field can be evaluated using the vector field samples (``arrows'') $\vec v_n \in \mathbb R^d$ in our training dataset, viz.
\begin{equation}
    \label{eqn:f_op_data_driven}
    \mathcal F^{(\alpha)}\vec V = \sum_{(k,l)\in \hat{\mathbb I}_{N,N}} \sum_{n=1}^N F_k^{(\alpha)} \cdot \vec v_n \phi_l^{(\alpha)}(x_n) w_\alpha(x_n) \hat\varepsilon_{kl}.
\end{equation}

Combining~\cref{eqn:g_d_ops_data_driven} and~\cref{eqn:f_op_data_driven}, and using projections $\pi_L \colon \ell^2(\mathbb I_N) \to \mathbb \ell^2(\mathbb I_N)$ and $\hat\pi_{L_1,L_2} \colon \ell^2(\hat{\mathbb I}_{N, N}) \to \ell^2(\hat{\mathbb I}_{N,N})$ and spectral resolution parameters $\ell = (L,L_1, L_2, \eta, \LD)$ as in \cref{subsec:projection}, we arrive at the regression problem
\begin{equation}
    \label{eqn:b_regression_datadriven}
    G^{(L,\eta,\LD, \alpha)} b^{(\ell,\alpha)} = \mathcal D^{(L, \LD,\alpha)} \mathcal F^{(L_1,L_2,\alpha)} \vec V,
\end{equation}
where $\mathcal F^{(L_1,L_2,\alpha)} = \hat \pi_{L_1, L_2} \mathcal F^{(\alpha)}$. This regression problem is a data-driven analog of~\eqref{eqn:b_regression_trunc_2}.  The solution of \eqref{eqn:b_regression_datadriven},
\begin{equation} \label{eqn:def:b_ell_eta_alpha}
    b^{(\ell,\alpha)} =  G^{(L, \LD, \eta,\alpha),+}\mathcal D^{(L, \LD,\alpha)} \mathcal F^{(L_1,L_2,\alpha)} \vec V,
\end{equation}
is our data-driven approximation of $b^{(\ell)}$ from~\cref{eqn:minim_2} with coefficients $b_{ij}^{(\ell,\alpha)} = \langle \varepsilon_{ij}, b^{(\ell,\alpha)}\rangle_{\ell^2(\mathbb I)}$. Finally, defining $B_{ij}^{(\LD,\alpha)} \colon L^2(\nu_\alpha) \to L^2(\nu_\alpha)$ to have matrix elements
\begin{displaymath}
    \langle \phi_l^{(\alpha)}, B_{ij}^{(\LD,\alpha)}\phi_k\rangle_{L^2(\nu_\alpha)} = d_{ijkl}^{(\LD,\alpha)},
\end{displaymath}
and setting $B_{ij}^{(L_1,L_2,\LD, \alpha)} = \Pi_{L_1,N} B_{ij}^{(\LD,\alpha)} \Pi_{L_2,N}$, leads to our data-driven approximation $\vec V^{(\ell,\alpha)} \in C^\infty(\mathbb R^d; \mathbb R^d)$ of $\vec V^{(\ell)}$ from~\cref{eqn:VL2},
\begin{equation}
    \label{eqn:def:V_ell_eta_alpha}
    \vec V^{(\ell,\alpha)} = \sum_{i=0}^{L-1}\sum_{j=1}^J \sum_{k=0}^{L_1-1}\sum_{l=0}^{L_2 - 1} d_{ijkl}^{(\LD,\alpha)} b_{ij}^{(\ell,\alpha)} F_k^{(\alpha)}\varphi_l^{(\alpha)}.
\end{equation}

Each of the terms in \cref{eqn:def:V_ell_eta_alpha} is data-driven and computable. The schematic in \cref{fig:Time_complexity} includes information about the time complexity of training (i.e., computation of $d_{ijkl}^{(\LD,\alpha)}$, $b_{ij}^{(\ell,\alpha)}$, and $F_k^{(\alpha)}$) and prediction (i.e., evaluation of $\vec V^{(\ell,\alpha)}(y)$) with our scheme. In the next section, we study the convergence of the reconstructed vector field \eqref{eqn:def:V_ell_eta_alpha}.

\section{Convergence analysis} \label{sec:convanalysis}
We study the asymptotic consistency of the data-driven, SEC-based approximation of the dynamical vector field (\cref{subsec:convergence_vector_field}) and the accuracy of the associated dynamical orbits (\cref{subsec:orbitrecon}).

\subsection{Consistency of vector field approximation}\label{subsec:convergence_vector_field} Our main approximation result establishes convergence of the data-driven vector field approximation scheme from \cref{subsec:secalgo} to the true vector field $\vec V$ in $L^2$ sense in an asymptotic limit of large data and infinite spectral resolution.

\begin{theorem} \label{thm:1}
    Let $\vec{V}^{(\ell, \alpha)}$ be the data-driven vector field \cref{eqn:def:V_ell_eta_alpha}, where $\alpha = (\epsilon, N)$, $N$ is the training dataset size,  $\epsilon$ is the kernel bandwidth parameter and $\ell = \paran{L, L_1, L_2, \eta, \LD}$ is the collection of SEC spectral resolution parameters. Let \cref{A:1,A:2,A:3,A:4} hold, where the large-data limit in \cref{A:4} is taken along a sequence $\alpha_N = (\epsilon_N, N)$ tending to $(0^+, \infty)$. Then,
    \begin{displaymath} 
        \lim_{\ell \to \bm\infty} \lim_{N \to \infty} \norm{ \vec{V}^{(\ell, \alpha_N)} - \vec V }_{L^2(\nu_*; \mathbb R^d)} = 0,
    \end{displaymath}
    where $\lim_{\ell \to\bm\infty}$ denotes the iterated limit $\lim_{L\to\infty}\lim_{L_1 \to \infty}\lim_{L_2 \to \infty}\lim_{\eta\to 0^+}\lim_{\LD \to \infty}$.
\end{theorem}

\paragraph{Proof of \cref{thm:1}} The convergence claim in \cref{thm:1} involves two types of limits: (i) the large-data limit $\alpha_N \to (\infty, 0^+)$; and (ii) infinite spectral resolution, $\ell \to \bm\infty$. We study these limits separately in the ensuing subsections.

\subsubsection{Large-data limit}

We begin by showing that the weighted empirical measures $\nu_{\alpha}$ from~\eqref{eqn:weighted_measure} consistently approximate the Riemannian volume measure $\nu$.

\begin{lemma}
    \label{lem:vol_measure_approx}
    Under \cref{A:3,A:4} and for the normalization of the weights $w_\alpha$ in~\eqref{eqn:weight_normalization} and~\eqref{eqn:vol_est}, the measures $\nu_{\alpha_N}$ converge weakly to $\nu$,
    \begin{displaymath}
        \lim_{N\to\infty} \int_{\mathcal M} f \, d\nu_{\alpha_N} = \int_{\mathcal M} f \, d\nu, \quad \forall f \in C(\mathcal M).
    \end{displaymath}
\end{lemma}

\begin{proof}
    By construction, the operator $P_\alpha$ is self-adjoint on $L^2(\nu_{\alpha})$ and admits the decomposition $P_\alpha = \sum_{j=0}^{N-1} \Lambda_j^{(\alpha)} \phi_j^{(\alpha)} \langle \phi_j^{(\alpha)}, \cdot\rangle_{L^2(\nu_\alpha)}$. This implies that the kernel $p_\alpha$ admits the Mercer-type expansion $p_\alpha(x_m, x_n) = \sum_{j=0}^{N-1} \Lambda_j^{(\alpha)} \phi_j^{(\alpha)}(x_m)\phi_j^{(\alpha)}(x_n)w_\alpha(x_n)$, giving
    \begin{displaymath}
        \tr_{L^2(\nu_\alpha)} P_\alpha = \sum_{j=0}^{N-1}\langle \phi_j^{(\alpha)}, P_\alpha \phi_j^{(\alpha)}\rangle_{L^2(\nu_\alpha)} = \int_{\mathcal M} p_\alpha(x, x)\, d\mu_N(x) = \tr \bm P_\alpha.
    \end{displaymath}
    By \cref{A:4}(i, iii), we have $\lim_{N\to\infty} \epsilon_N^{\dim(\mathcal M)}  \left(\tr_{L^2(\nu_{\alpha_N})} P_{\alpha_N} - \tr_{L^2(\nu)} e^{-\epsilon_N^2 \Delta}\right) = 0$, where $e^{-\epsilon_N^2\Delta}$ is the time-$\epsilon_N^2$ heat operator on $L^2(\nu)$. Meanwhile, by the Minakshisundaram--Pleijel heat-trace formula (see \cref{app:dm}) we have $\lim_{N\to\infty} (4\pi \epsilon_N^2)^{\dim(\mathcal M)/2} \tr_{L^2(\nu)} e^{-\epsilon_N^2 \Delta} = \nu(\mathcal M)$. As a result, under the normalization from~\eqref{eqn:weight_normalization} and~\eqref{eqn:vol_est}, the weights $w_\alpha$ satisfy
    \begin{equation}
        \label{eqn:normalization_lim}
        \lim_{N\to\infty}\int_{\mathcal M} w_{\alpha_N} d\mu_N = \nu(\mathcal M).
    \end{equation}
    Now, by \cref{A:3} and \cref{A:4}(ii) (and noting~\eqref{eqn:phi_innerprod_alpha} and~\eqref{eqn:phi_innerprod_alpha}), we have
    \begin{displaymath}
        \lim_{N\to\infty}\int_{\mathcal M} f \, d\nu_\alpha = C \int_{\mathcal M} f \frac{1}{\sigma} \, d\mu = C \int_{\mathcal M} f \, d\nu, \quad \forall f \in C(\mathcal M),
    \end{displaymath}
    where $C$ is a normalization constant. By~\eqref{eqn:normalization_lim}, we have $C=1$ and the claim of the lemma follows.
\end{proof}

\Cref{lem:vol_measure_approx}, in conjunction with the uniform convergence of the eigenvectors in \cref{A:4}(ii), implies that the data-driven $c$-coefficients from~\cref{subsec:secalgo} converge. Indeed, for every $i,j,p \in \mathbb N$ and sufficiently large $N$ we have
\begin{align*}
    \lvert c_{ijp}^{(\alpha)} - c_{ijp}\rvert
    &= \lvert \langle \phi_p^{(\alpha)}, \phi_i^{(\alpha)} \phi_j^{(\alpha)}\rangle_{L^2(\nu_a)} - \langle \phi_p, \phi_i \phi_j\rangle_{L^2(\nu)}\rvert \\
    &= \lvert \nu_\alpha(\phi_p^{(\alpha)}\phi_i^{(\alpha)}\phi_j^{(\alpha)}) - \nu(\phi_p\phi_i\phi_p)\rvert \\
    &\leq \lvert \nu_\alpha(\phi_p^{(\alpha)}\phi_i^{(\alpha)}\phi_j^{(\alpha)} - \phi_p\phi_i\phi_j)\rvert + \lvert (\nu_\alpha - \nu)(\phi_p \phi_i \phi_j)\rvert,
\end{align*}
and both terms in the last line converge to 0 on the sequence $\alpha_N$ as $N \to \infty$. Using \cref{A:4}(i), we can therefore deduce convergence of the $g$- and $d$- coefficients (the latter, at fixed $\LD$),
\begin{equation}
    \label{eqn:coeffs_lim}
    \lim_{N\to\infty} g_{pjk}^{(\alpha_N)} = g_{pjk}, \quad \lim_{\alpha_N \to \infty} d_{ijkl}^{(\LD,\alpha)} = d_{ijkl}^{(\LD)}, \quad \forall i,j,k,l,p \in \mathbb N.
\end{equation}

\Cref{lem:vol_measure_approx} and continuity of $F$ (\cref{A:2}) also imply convergence of the expansion coefficients of the embedding, $\lim_{N\to\infty}F_k^{(\alpha_N)} = F_k$ for every $k \in \mathbb N$, as well as convergence of each of the terms of the sequence $\mathcal F^{(\alpha_N)} \vec V $, i.e., $\lim_{N\to\infty} (\mathcal F^{(\alpha_N)}\vec V)_{kl} = (\mathcal F \vec V)_{kl}$ for all $k,l \in \mathbb N_0)$. The terms of that sequence and the $d_{ijkl}^{(\LD,\alpha)}$ coefficients completely determine the right-hand side of the regression problem in~\eqref{eqn:def:b_ell_eta_alpha}, so that
\begin{equation} %
    \label{eqn:regr_rhs_conv}
    \lim_{N\to\infty} (\mathcal D^{(L, \LD,\alpha_N)} \mathcal F^{(L_1,L_2,\alpha_N)}\vec V)_{ij} = (\mathcal D^{(L,\LD)} \mathcal F^{(L_1,L_2)}\vec V)_{ij}, \quad \forall (i, j) \in \mathbb I.
\end{equation}
Similarly, the $d_{ijkl}^{(\alpha_N)}$ coefficients determine the spectrally truncated Gram operators $G^{(L,\eta,\LD, \alpha_N)}$, giving
\begin{displaymath}
    \lim_{N\to\infty} \langle\varepsilon_{ij}, G^{(L,\eta,\LD, \alpha_N)} \varepsilon_{kl} \rangle_{\ell^2(\mathbb I)} = \lim_{N\to\infty} d_{ijlk}^{(\LD,\alpha_N)} = d_{ijlk}^{(\LD)} = \langle \varepsilon_{ij}, G^{(L,\eta,\LD)} \varepsilon_{kl}\rangle_{\ell^2(\mathbb I)},
\end{displaymath}
for all $(i,j),(k,l) \in \ell^2(\mathbb I)$, where $G^{(L,\eta,\LD)}$ was defined in \eqref{eqn:spectrally_truncated_gram}.

Next, we recall the following result on convergence of pseudoinverses of finite-rank operators (e.g., \cite[Theorem~4.2]{rakovcevic1997cont}).
\begin{lemma} \label{lem:pg9z}
    Suppose that $A_\tau$ is a family of operators on a finite-dimensional Hilbert space converging in norm as $\tau\to \infty$ to an operator $A$. Then
    \[ \lim_{\tau\to \infty} \norm{A^+_\tau - A^+} = 0 \quad \text{iff} \quad \lim_{\tau\to \infty} \rank A_\tau = \rank A. \]
\end{lemma}

Since $\eta>0$, the norm convergence of $G^{(L,\eta,\LD, \alpha)}$ to $G^{(L,\eta,\LD)}$ implies rank convergence, so by \cref{lem:pg9z} $G^{(L,\eta,\LD, \alpha),+}$ converges to $G^{(L,\eta,\LD),+}$ in the operator norm induced from $\ell^2(I_L)$. It follows that as $N\to\infty$ the solution vector $b^{(\ell,\alpha_N)}$ from \eqref{eqn:def:b_ell_eta_alpha} converges to $b^{(\ell)}$ from~\eqref{eqn:minim_2}. As a result, the data-driven vector field $\vec V^{(\ell,\alpha)}$ converges to $\vec V^{(\ell)}$ from~\eqref{eqn:VL2}; that is, $\lim_{N\to\infty} \vec V^{(\ell,\alpha_N)}(y) = \vec V^{(\ell)}(y)$, where the convergence is uniform for $y$ in compact subsets of $\mathbb R^d$.

\subsubsection{Limit of infinite spectral resolution}
We now show that $\vec V^{(\ell)} \equiv \vec V^{(L,L_1,L_2,\eta,\LD)}$ converges to $\vec V$ in $L^2(\nu_*; \mathbb R^d)$ in the iterated limit $L \to \infty$ after $L_1 \to \infty$ after $L_2 \to \infty$ after $\eta \to 0^+$ after $\LD \to \infty$. Since in \cref{subsec:projection} we have already shown that $V^{(L)}$ converges to $V$ as $L \to \infty$ in Hodge norm (see, in particular, \eqref{eqn:VL_to_V}), it is enough to show that $\vec V^{(\ell)}$ converges to $\vec V^{(L)}$ from~\eqref{eqn:FVL}.

First, by definition of the $d_{ijkl}^{(\LD)}$ coefficients we have $\lim_{\LD \to\infty} d_{ijkl}^{(\LD)} = d_{ijkl}$ for each $(i,j) \in \mathbb I$ and $(k, l)\in \hat{\mathbb I}$. This implies that
\begin{equation}
    \label{eqn:d_conv_Ld}
    \lim_{\LD \to\infty} \mathcal D^{(L,\LD)} \mathcal F^{(L_1,L_2)} \vec V = \mathcal D_L \mathcal F^{(L_1,L_2)} \vec V
\end{equation}
in $\ell^2(\mathbb I)$ norm. Moreover, letting $G^{(L,\eta)}$ be the spectrally truncated Gram operator defined analogously to $G^{(L,\eta,\LD)}$ in~\eqref{eqn:spectrally_truncated_gram} using the $d_{ijkl}$ coefficients instead of $d_{ijkl}^{(\LD)}$, we have that $G^{(L,\eta,\LD)}$ converges to $G^{(L,\eta)}$ in norm and rank (the latter since $\eta >0$) as $ L \to \infty$. Thus, by \cref{lem:pg9z} and~\eqref{eqn:d_conv_Ld}, it follows that the regression solution $b^{(\ell)}$ in~\eqref{eqn:minim_2} converges to $b^{(L,L_1,L_2,\eta)} = G^{(L,\eta), +} \mathcal D_L \mathcal F^{(L_1,L_2)}$.

Next, we consider the $\eta\to 0^+$ limit. Given a self-adjoint operator $A$ on a finite-dimensional Hilbert space $H$ with orthonormal eigendecomposition $A u_j = \alpha_j u_j$ and a function $f \colon \mathbb R \to \mathbb R$, $f(A)$ will denote the corresponding operator-valued function from the functional calculus, $f(A) = \sum_{j} f(\alpha_j) u_j \langle u_j, \cdot\rangle_H$. Note that $\chi_{(\eta,\infty)}(A)$ for the characteristic function $\chi_{(0,\infty)}$ of $(0,\infty)$ gives the spectral truncation of $A$ analogously to~\eqref{eqn:spectrally_truncated_gram}. We will use the shorthand notation $A_\eta \equiv \chi_{(\eta,\infty)}(A)$.

\begin{lemma} \label{lem:jrd03l}
    Let $A$ be a positive operator on a finite-dimensional Hilbert space and $A_\eta$ the corresponding spectrally truncated operators for $\eta>0$. Then, we have convergence of the respective pseudoinverses, $\lim_{\eta\to 0^+} \norm{ A^+_\eta - A^+ } = 0$.
\end{lemma}

\begin{proof}
    Let $\alpha_*>0$ be strictly smaller than the smallest strictly positive eigenvalue of $A$. Fix a continuous function $f \colon \mathbb R \to \mathbb R$ such that $f(\eta) = 1 / \eta$ for $\eta \in (\alpha_*, \infty)$ and $f(\eta) = 0$ for $\eta \leq 0$. Then, we have $f(A) = A^+$ and there exists $\eta_* >0$ such that $f(A_\eta) = A_\eta^+$ for all $ \eta \in (0,\eta_*)$. It then follows from the continuous functional calculus that
    \begin{displaymath}
        \lim_{\eta\to 0^+} A_\eta^+ = \lim_{\eta\to 0^+} f(A_\eta) = f\left(\lim_{\eta\to 0^+} A_\eta\right) = f(A) = A^+,
    \end{displaymath}
    where the limits are taken in operator norm.
\end{proof}

Applying \cref{lem:jrd03l} for the spectrally truncated approximations $G^{(L,\eta)}$ of the Gram operator $G_L$, we deduce that as $\eta \to 0^+$, $G^{(L,\eta),+}$ converges to $G^+_L$ and thus that $b^{(L,L_1,L_2,\eta)}$ converges to $b^{(L,L_1,L_2)} := G^+_L \mathcal D_L \mathcal F^{(L_1,L_2)} \vec V$. In this limit, our approximate vector field, $\vec V^{(L,L_1,L_2)} \in C^\infty(\mathbb R^d; \mathbb R^d)$ is given by
\begin{equation}
    \label{eqn:FVL2}
    \vec V^{(L,L_1,L_2)} = \sum_{i=0}^{L-1} \sum_{j=1}^J \sum_{k=0}^{L_1-1} \sum_{l=0}^{L_2-1} d_{ijkl} b_{ij}^{(L,L_1,L_2)} F_k \varphi_l
\end{equation}
Note that this is reconstruction is close to $\vec V^{(\ell)}$ in \cref{eqn:VL2}, but without the use of the $\LD$-truncation for the $d$-coefficients and $\eta$-regularization.

Next, observe that $\vec V^{(L,L_1,L_2)}(F(x)) = \sum_{i=0}^{L-1} \sum_{j=1}^J B_{ij}^{(L_1,L_2)} F(x)$ for all $x \in \mathcal M$. We use regularity of the SEC frame elements and the embedding $F$ to show that $B_{ij}^{(L_1,L_2)} F$ converges in $L^2(\nu; \mathbb R^d)$ as $L_1,L_2 \to \infty$.

\begin{lemma}
    \label{lem:frame_op_bound}
    The frame elements $ \{ B_{ij} \}_{(i,j)\in \mathbb I}$ are uniformly norm-bounded as operators from $C^1(\mathcal M)$ to $L^2(\nu)$.
\end{lemma}

\begin{proof}
    The claim follows by direct verification. With $f \in C^1(\mathcal M)$, we have
    \begin{align*}
        \lVert B_{ij} \triangleright f\rVert^2_{L^2(\nu)}
        &= \langle \phi_i^2, (\nabla \phi_j \cdot \nabla f)^2\rangle_{L^2(\nu)} \leq \lVert \phi_i^2\rVert_{L^1(\nu)} \lVert \nabla \phi_k \cdot \nabla f\rVert^2_{L^\infty(\nu)} =  \lVert \nabla \phi_k \cdot \nabla f\rVert^2_{L^\infty(\nu)} \\
        &\leq \lVert \phi_j \rVert_{C^1(\mathcal M)} \lVert f\rVert_{C^1(\mathcal M)}^2 \leq M \lVert \lVert f \rVert^2_{C^1(\mathcal M)},
    \end{align*}
    where $M = \max \{\lVert \phi_j \rVert_{C^1(\mathcal M)}\}_{j=1}^J$.
\end{proof}

Since $F$ is $C^\infty$, $\Pi_{L_2} F$ converges to $F$ componentwise in $C^1(\mathcal M)$ norm. Thus, by \cref{lem:frame_op_bound},
\begin{equation}
    \label{eqn:bij_f_l1_l2_conv}
    \lim_{L_1\to\infty}\lim_{L_2 \to \infty} B_{ij}^{(L_1,L_2)} F = \lim_{L_1 \to \infty}\lim_{L_2 \to \infty}\Pi_{L_1}B_{ij}\Pi_{L_2} F = \lim_{L_1 \to \infty} \Pi_{L_1} B_{ij} F = B_{ij}F,
\end{equation}
and the convergence is uniform with respect to $(i,j) \in \mathbb I$.

Using similar arguments, one can verify that $\mathcal D_L \mathcal F$ is bounded as an operator from $C^1(Y; \mathbb R^d)$ to $\ell^2(\mathbb I)$, giving $\lim_{L_1\to \infty}\lim_{L_2 \to \infty}  \mathcal D_L \mathcal F^{(L_1, L_2)} \vec V = \mathcal D_L \mathcal F \vec V$, and thus
\begin{equation}
    \label{eqn:b_l1_l2_conv}
    \lim_{L_1 \to \infty}\lim_{L_2\to \infty} b^{(L,L_1,L_2)} = b^{(L)}.
\end{equation}
Using~\eqref{eqn:bij_f_l1_l2_conv} and~\eqref{eqn:b_l1_l2_conv} in~\eqref{eqn:FVL2}, we conclude that
\begin{displaymath}
    \lim_{L_1\to\infty}\lim_{L_2\to\infty} \vec V^{(L,L_1,L_2)} = \vec V^{(L)}
\end{displaymath}
in $L^2(\nu*; \mathbb R^d)$. This completes our proof of \cref{thm:1}. \hfill \qed

\subsection{Orbit reconstruction}
\label{subsec:orbitrecon}

As discussed in \cref{sec:dynamicslearn}, for every $x \in \mathcal M$, the SEC-approximated vector field $\vec V^{(L)}$ from~\cref{eqn:V_pushforward} induces a dynamical trajectory $t \mapsto y^{(L)}(x, t) \in \mathbb R^d$ in data space, where $y^{(L)}(x, \cdot) \in C^1(\mathbb R; \mathbb R^d)$ is the solution of the initial-value problem
\begin{equation}
    \label{eqn:ivp}
    \dot y^{(L)}(x, t) = \vec V^{(L)}(y^{(L)}(x,t)), \quad y^{(L)}(x,0) = F(x).
\end{equation}
In this subsection, we assess the error of such orbits relative to orbits $t \mapsto y(x, t)$ generated by the true vector field $V$ on the manifold, i.e., $y(x,t) = F(\Phi^t(x))$. Orbits generated by the vector fields $\vec V^{(\ell)}$ and $\vec V^{(\ell,\alpha)}$ from~\eqref{eqn:VL2} and~\eqref{eqn:def:V_ell_eta_alpha}, respectively, which involve additional approximations compared to $V^{(L)}$, can be treated similarly and obey analogous bounds as orbits under $V^{(L)}$ in the appropriate asymptotic limits.

As a measure of expected error over initial conditions, we will use the norm
\begin{displaymath}
    \lVert f \rVert_{L^1(\nu; \mathbb R^d)} := \int_{\mathcal M}\lVert f(x)\rVert_1 \, d\nu(x) \equiv \sum_{j=1}^d \lVert f_j\rVert_{L^1(\nu)},
\end{displaymath}
where $f = (f_1,\ldots, f_d) \colon \mathcal M \to \mathbb R^d$ has components $f_j \in L^1(\nu)$. We will also make use of the following seminorms on $C(\mathcal M; \mathbb R^d)$,
\begin{displaymath}
    \lvert f \rvert_{1,2} = \sum_{j=1}^d \lVert \nabla f_j\rVert_{\mathbb H} \equiv - \sum_{j=1}^d \langle f_j, \Delta f_j\rangle_{L^2(\nu)}, \quad \lvert f \rvert_{2,\infty} = \sum_{j=1}^d \left\lVert \Delta f_j  \right \rVert_{L^\infty(\nu)}.
\end{displaymath}
With these definitions, we have the following bound on the expected orbit distance over initial conditions based on the distance between two vector fields in the Hodge norm.

\begin{proposition} \label{prop:L2_vector_approx}
    Assume $\calM$, $V$ as in \cref{A:1}. Let $\hat{V} \in C(\mathcal M, TM)$ be any other continuous vector field on $\calM$ generating a flow $\hat\Phi^t \colon \mathcal M \to \mathcal M$ and associated data space trajectory map $\hat y(x, t) = F(\hat\Phi^t(x))$. Fix $T>0$ and set
    \begin{displaymath}
        \alpha := \left\lVert V - \hat V \right\rVert_{\mathbb H} \max_{t \in [0, T]} \left\lvert F \circ \Phi^t \right\rvert_{1,2}, \quad \beta = \lVert \hat V\rVert_{\mathbb H} \max_{t \in [0, T]} \left\lvert F \circ \Phi^t - F \circ \hat \Phi^t \right\rvert_{2,\infty}.
    \end{displaymath}
    Then, for every $t \in [0, T]$, the error norm $\varepsilon(t) = \left\lVert y(t, \cdot) - \hat y(t, \cdot)\right\rVert_{L^1(\nu; \mathbb R^d)}$ satisfies the following bounds on growth rate and value:
    \begin{displaymath}
        \dot \varepsilon(t) \leq \alpha + \beta \sqrt{\epsilon(t)}, \quad \varepsilon(t) \leq t \left( \sqrt{\alpha} + \frac{\beta \sqrt{t}}{2}\right)^2.
    \end{displaymath}
\end{proposition}

\begin{proof}
    See \cref{app:proof_orbit_recon}.
\end{proof}

In the case of the SEC-approximated vector field, $\hat V \equiv V^{(L)}$, the Hodge norm convergence \cref{eqn:VL_to_V} implies that the corresponding error term $\alpha \equiv \alpha_L$ in \cref{prop:L2_vector_approx} can be made arbitrarily small by increasing the spectral resolution parameter $L$. On the other hand, convergence of $V^{(L)}$ to $V$ in Hodge norm does not, in general, provide control of $\beta \equiv \beta_L$ as $L$ increases. Thus, we cannot conclude from \cref{eqn:VL_to_V} and \cref{prop:L2_vector_approx} alone that the expected error $\varepsilon(t)$ can be made arbitrarily small by increasing $L$ on a fixed time interval $[0, T]$. Nevertheless, we can still conclude that the short-term error and error growth rate can be made arbitrarily small. In particular, for $T_L = \min \{ T, 2 \sqrt{\alpha_L}/ \beta_L \}$ and $t \in [0, T_L]$ we have $\varepsilon(t) \leq 2 \sqrt{\alpha_L} t$. Since $\alpha_L \to 0$ as $L \to \infty$, the expected error $\varepsilon(t)$ can be made arbitrarily small on an $L$-independent time interval if $\liminf_{L\to \infty}T_L >0$.

\section{Numerical examples}
\label{sec:numexp}
We apply the SEC framework described in \crefrange{sec:probstatement}{sec:data_driven} to various dynamical systems on the unit circle $S^1$ and 2-torus $\mathbb T^2$, viewed as embedded submanifolds of $\mathbb R^2$ and $\mathbb R^3$ under standard (smooth) embeddings, respectively.

In each example, we use training data $ \{ y_n \in \mathbb R^d \}_{n=1}^N$ and $ \{ \vec v_n \in \mathbb R^d \ \}$ for the embedded manifold and dynamical vector field, respectively, generated from a set $ \{ x_n \in \mathcal M \}_{n=1}^N$ of points on the manifold lying in a uniform grid. Using the $y_n$ and $\vec v_n$ data (but assuming no knowledge of the underlying manifold points $x_n$), we apply the diffusion maps algorithm (\cref{algo:dffsn}) to compute the eigenpairs $(\lambda_j^{(\alpha)}, \phi_j^{(\alpha)})$ of the Laplacian, and build the SEC approximation $V^{(\ell,\alpha)}$ of the vector field using \cref{algo:bcoeff}. We then reconstruct the pushforward vector field $\vec V^{(\ell,\alpha)}$ in data space using \eqref{eqn:def:V_ell_eta_alpha} and solve the associated initial-value problem \eqref{eqn:ivp}.

We assess the performance of our scheme quantitatively by computing normalized mean square errors ($R^2$ coefficients) of the reconstructed vector fields,
\begin{displaymath}
    R^2 = \frac{\sum_{n=1}^N\lVert \vec V(x_n) - \vec V^{(\ell,\alpha)}(x_n)\rVert^2_2}{\sum_{n=1}^N \lVert \vec V(x_n)\rVert^2_2},
\end{displaymath}
which provide estimates of the SEC approximation accuracy with respect to the Hodge norm on vector fields. We also assess performance qualitatively by comparing dynamical trajectories in embedding space $\mathbb R^d$ under the true vector field $\vec V$ and the SEC approximation $\vec V^{(\ell,\alpha)}$. The latter are given by solutions of the initial-value problem (cf.\ \eqref{eqn:ivp})
\begin{displaymath}
    \dot y^{(\ell,\alpha)}(t) = \vec V^{(\ell,\alpha)}(y^{(\ell,\alpha)}(t)), \quad y^{(\ell,\alpha)}(0) = y_0 \in \mathbb R^d.
\end{displaymath}
In our experiments, we will take the initial condition $y_0$ to lie in the embedded manifold $F(\mathcal M)$, i.e., $y_0 = F(x_0)$ for some $x_0 \in \mathcal M$. The true dynamical trajectory that we compare against is thus given by $y(t) = F(x(t)),$ where $x(t) = \Phi^t(x_0)$ solves the initial-value problem
\begin{equation}
    \label{eqn:ivp_manifold}
    \dot x(t) = V(x(t)), \quad x(0)= x_0
\end{equation}
on the manifold $\mathcal M$ (see \cref{subsec:orbitrecon}).

The dataset attributes and approximation parameters used in our experiments are summarized in \cref{tab:expt}.

\begin{table}
    \centering
    \caption{Dataset attributes and numerical approximation parameters used in our experiments.}
    \small
    \label{tab:expt}
    \begin{tabular}{lcc}
        & Circle examples & 2-Torus examples\\
        \cline{2-3}
        Manifold $\mathcal M$ & $S^1$ & $\mathbb T^2$ \\
        Data space & $\mathbb R^2$ & $\mathbb R^3$ \\
        Embedding map $F$ & \cref{eqn:circleembedding} & \cref{eqn:torusembedding} \\
        Training dataset size $N$ & 800 & $150 \times 150 = \text{22,500}$\\
        Kernel bandwidth $\epsilon$ & 0.2 & 0.0966 \\
        Number of gradient fields $J$  & 10 & 20 \\
        SEC discretization parameters $L$, $L_1$ & 20 & 120 \\
        SEC discretization parameters $\LD$, $L_2$ & 40 & 280\\
        Gram operator spectral truncation parameter $\eta$ & 0.01 & 0.05 \\
        \hline
    \end{tabular}
\end{table}

\subsection{Dynamical systems on the circle}
\label{subsec:circlerotation}

Every continuous vector field $V$ on $S^1$ is of the form
\begin{equation}
    \label{eqn:s1vecdef}
    V(\theta) = h(\theta) \frac{\partial\ }{\partial\theta},
\end{equation}
where $h$ is a continuous function, $\theta \in [0, 2\pi)$ is an angle coordinate, and $\frac{\partial\ }{\partial\theta} \in T_\theta(S^1)$ the associated coordinate tangent vector. Here, we consider three vector fields with distinct dynamical behavior:
\begin{itemize}
    \item A uniform rotation, $h(\theta) = 1$.
    \item A dynamical flow with fixed points, $h(\theta) = 1 + c\cos\theta$ for a constant $c>1$, set here to $c= 1.5$.
    \item A variable-speed rotation, with $h(\theta) = e^{c \cos\theta}$ for a constant $c\neq 0$, set here to $c = 0.5$.
\end{itemize}
We denote the corresponding vector fields as $V_1$, $V_2$, and $V_3$, respectively. Note that $V_2$ has two unstable equilibria given by the solutions of $h(\theta) = 0$. For our chosen value of $c$ these turn out to be
\begin{equation}
    \label{eqn:fixedpoints}
    \theta \approx \text{2.30 rad}, \quad \theta \approx \text{3.98 rad}.
\end{equation}
The resulting dynamical flow is monotonic and takes place along a circular arc connecting these equilibria. Meanwhile, the variable-speed rotation under $V_3$ attains its maximum speed $e^c \approx 1.65$ at $\theta = 0$ and its minimum speed $e^{-c} \approx 0.61$ at $\theta=\pi$.

In all three cases, we use the standard embedding $F \colon S^1 \to \mathbb R^2$, given by
\begin{equation}
    \label{eqn:circleembedding}
    F(\theta) =
    \begin{pmatrix}
        \cos\theta \\
        \sin\theta
    \end{pmatrix}.
\end{equation}
The associated pushforward of the vector field from~\eqref{eqn:s1vecdef} is
\begin{displaymath}
    \vec{V}(y) = F_*V(y) = VF(\theta) = h(\theta)\frac{\partial}{\partial \theta} \begin{pmatrix} \cos{\theta}\\
    \sin{\theta}
\end{pmatrix} = h(\theta)\begin{pmatrix} \sin{\theta}\\
    -\cos{\theta}
\end{pmatrix},
\end{displaymath}
where $y = F(\theta)$. For this embedding of the circle, the eigenvalue problem~\eqref{eqn:lapl_eig} for the Laplacian can be analytically solved, giving eigenvalues $\lambda_j = j^2$ and
\begin{displaymath}
    \phi_j(\theta)
    =
    \begin{cases}
        1/\sqrt{2\pi},  & j = 0 ,\\
        \sin{\left(\frac{(j+1)\theta}{2}\right)} / \sqrt{\pi},  & j \ \text{is odd},\\
        \cos{\left(\frac{j\theta}{2}\right)} / \sqrt{\pi}, & j \ \text{is even}.
    \end{cases}
\end{displaymath}
for corresponding eigenfunctions normalized to unit norm on $L^2(\nu)$. Consequently, $F$ can be expressed componentwise as a finite linear combination of Laplace--Beltrami eigenfunctions,
\begin{displaymath}
    F(\theta) = \sqrt{\pi}
    \begin{pmatrix}
        \phi_2(\theta)\\
        \phi_1(\theta)
    \end{pmatrix}.
\end{displaymath}

Our training datasets are based on $N=800$ equidistant samples obtained from the embedding~\eqref{eqn:circleembedding} and the vector fields $V_1$, $V_2$, and $V_3$. We perform diffusion maps using the kernel bandwidth parameter $\epsilon = 0.2$. In all three cases, we use $J=10$ gradient fields $\nabla\phi_j$ to build SEC frames and set the remaining SEC spectral resolution parameters to $L = L_1 = 20$ and $\LD = L_2 = 40$. These parameter values were selected from a suite of cross-validation experiments. Even though $L$, $\LD$, $L_1$, and $L_2$ are independent discretization parameters, experimentally we found that fixing $L=L_1$ and $\LD = L_2$ yields accurate reconstruction results while simplifying the parameter tuning process. The $R^2$ coefficients for our chosen parameter values are 0.999731, 0.999220, and 0.999632, respectively for $V_1$, $V_2$, and $V_3$.

\Cref{fig:tangvecs1} shows a comparison of the true vs.\ SEC-approximated vector fields for the three examples, visualized as quiver plots on representative training points on the unit circle in $\mathbb R^2$. As expected from the large $R^2$ coefficients, the SEC approximations are seen to accurately reconstruct the true vector fields, exhibiting a high degree of tangency on the unit circle.

\begin{figure}
  \centering
  \includegraphics[width=10cm]{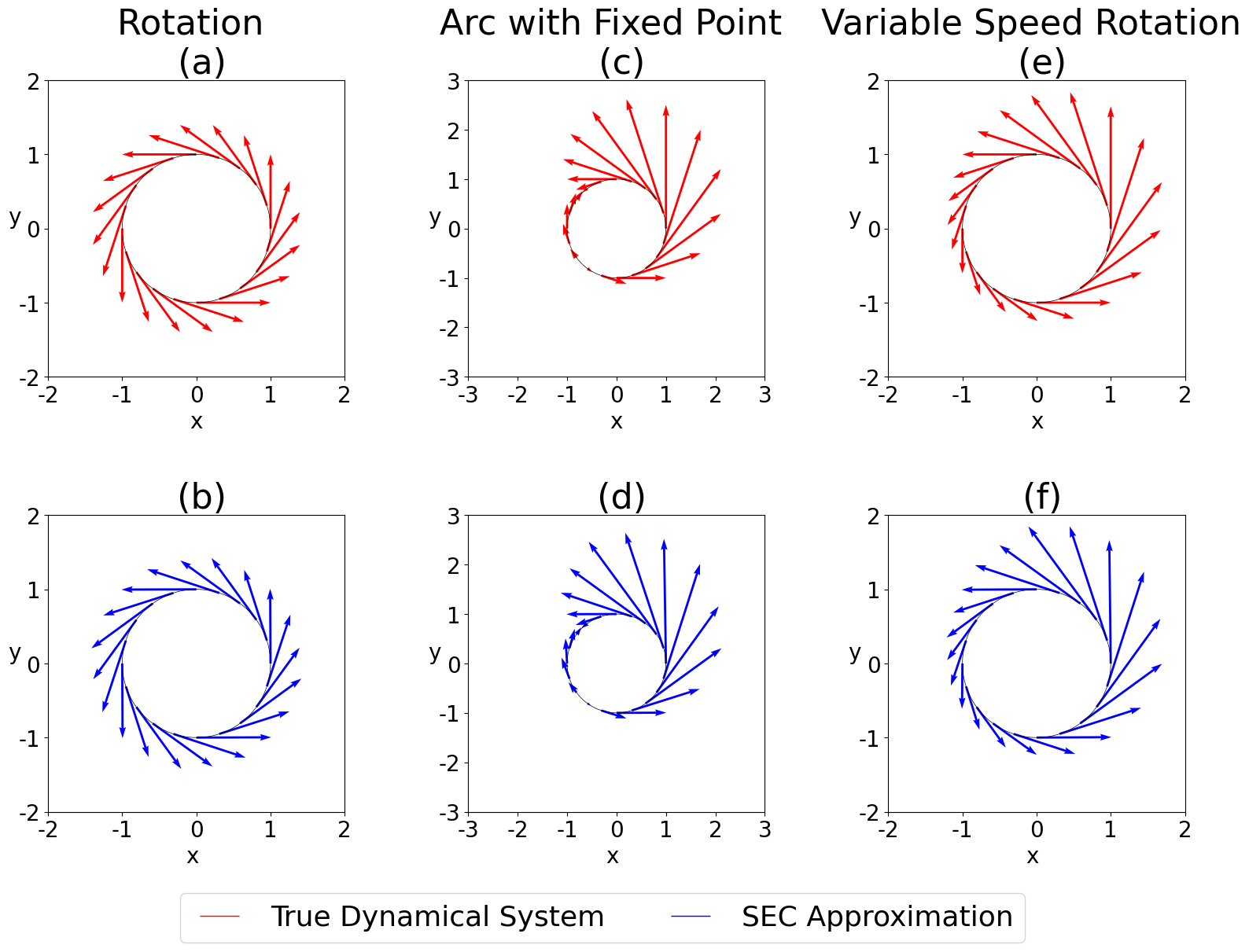}
  \label{fig:tangvecs1}
  \caption{Quiver plots of the true ($\vec V$; a, c, e) and SEC-approximated vector fields ($\vec V^{(\ell,\alpha)}$; b, d, f) on the unit circle embedded in $\mathbb R^2$. (a, b) Circle rotation, $V_1$; (c, d) connecting arc, $V_2$; and (e, f) variable-speed rotation, $V_3$. The panels show tangent vectors on the unit circle in $\mathbb R^2$ for a subset of points in the training dataset.}
\end{figure}

Recall that, unlike the true pushforward vector field $\vec V$, the SEC approximation $\vec V^{(\ell,\alpha)}$ can be evaluated at arbitrary points in $\mathbb R^2$, including points which do not lie in the embedded manifold $Y = F(\mathcal M)$. In \cref{fig:vFs1} we plot these out-of-sample extensions on the rectangular domain $[-5 , 5] \times [-5, 5]$ that contains $Y$. Intriguingly, the vector fields $\vec V^{{(\ell,\alpha)}}$ appear to consistently extend the angular structure of the true vector fields; that is, we have $\vec V^{(\ell,\alpha)}(y) \approx a(r) \vec V(F(\theta))$ for $y = r(\cos\theta,\sin,\theta)$ and a near-constant amplitude function $a(r)$. This behavior suggests favorable generalization behavior of the SEC approximations to points outside the support of the sampling distribution in data space. This could aid, e.g., orbit stability of~\eqref{eqn:ivp} under perturbations of the initial condition and/or error due to numerical timestepping.

\begin{figure}
  \centering
  \includegraphics[width=12cm]{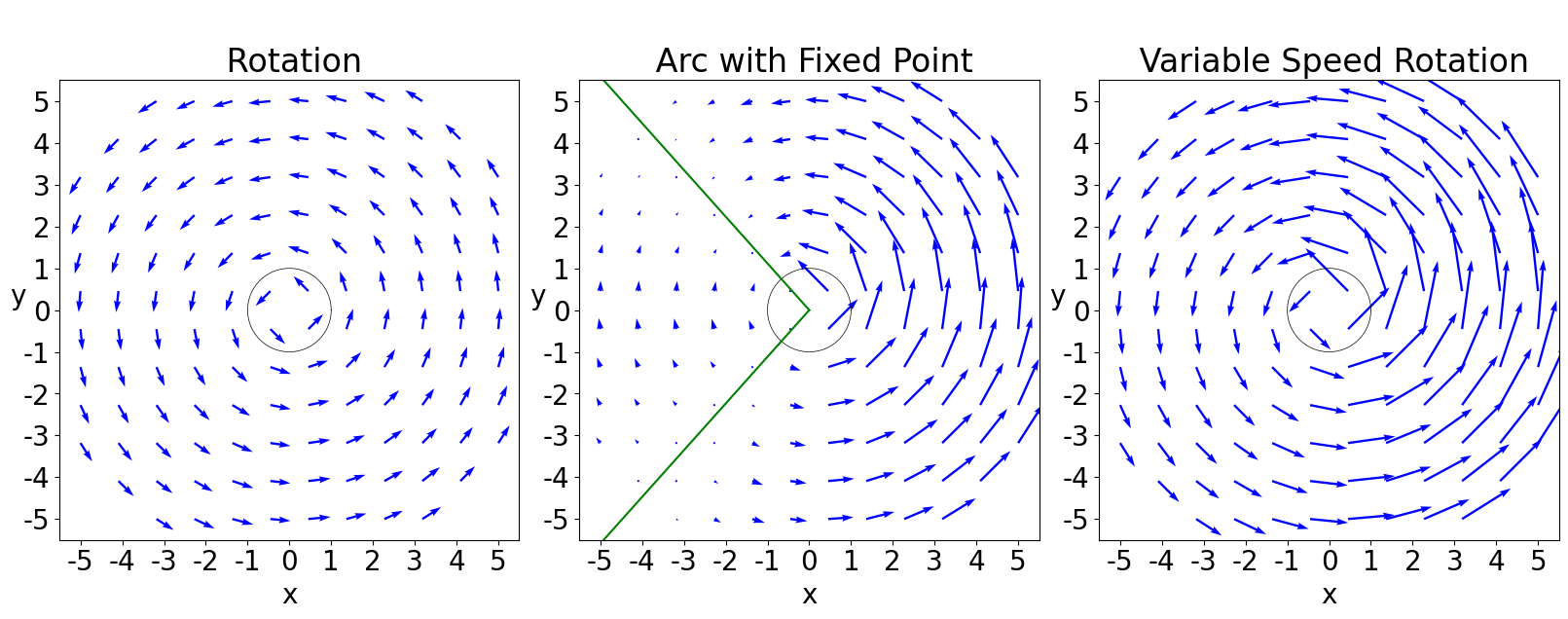}
  \label{fig:vFs1}
  \caption{As in \cref{fig:tangvecs1}, but for quiver plots of the SEC-approximated  vector fields $\vec V^{(\ell, \alpha)}$ on the box $[-5, 5] \times [-5, 5] \subset \mathbb R^2$. The unit circle $Y = F(\mathcal M)$ supporting the sampling distribution of the training data is drawn in each panel for reference. The green lines in panel (b) emanating from the origin depict the angular coordinates of the fixed points of $V_2$ from \eqref{eqn:fixedpoints}. Notice that $\vec V^{(\ell,\alpha)}$ consistently extends the pushforward $\vec V$ from $Y$ to other locations in $\mathbb R^2$.}
\end{figure}

Next in, \cref{fig:odesols1,fig:odetss1}, we examine representative dynamical trajectories $y(t)$ and $y^{(\ell,\alpha)}(t)$ under the true and SEC-approximated vector fields, respectively, visualized as trace plots in $\mathbb R^2$ (\cref{fig:odesols1}) and time series of the coordinates (\cref{fig:odetss1}). To obtain $y^{(\ell,\alpha)}(t)$ we solve~\eqref{eqn:ivp_manifold} for an arbitrary initial condition $y_0 = F(\theta_0)$. We obtain $y(t)$ by first solving~\eqref{eqn:ivp_manifold} (analytically, in the case of $V_1$, and numerically for $V_2$ and $V_3$) with initial condition $\theta_0 \in \mathcal M$. It is clear that the SEC provides a faithful reconstruction of the true dynamics in all three cases. It should be noted that qualitative reconstruction of connecting arcs under $V_2$ is non trivial; in particular, vector field approximations that do not exhibit fixed points would result in periodic trajectories that are qualitatively different from the true dynamics. In separate calculations, we have used the Newton method to verify that the SEC-approximation of $V_2$ does indeed exhibit fixed points (to numerical precision) in locations close to the true fixed points in~\eqref{eqn:fixedpoints}.

\begin{figure}
  \centering
  \includegraphics[width=10cm]{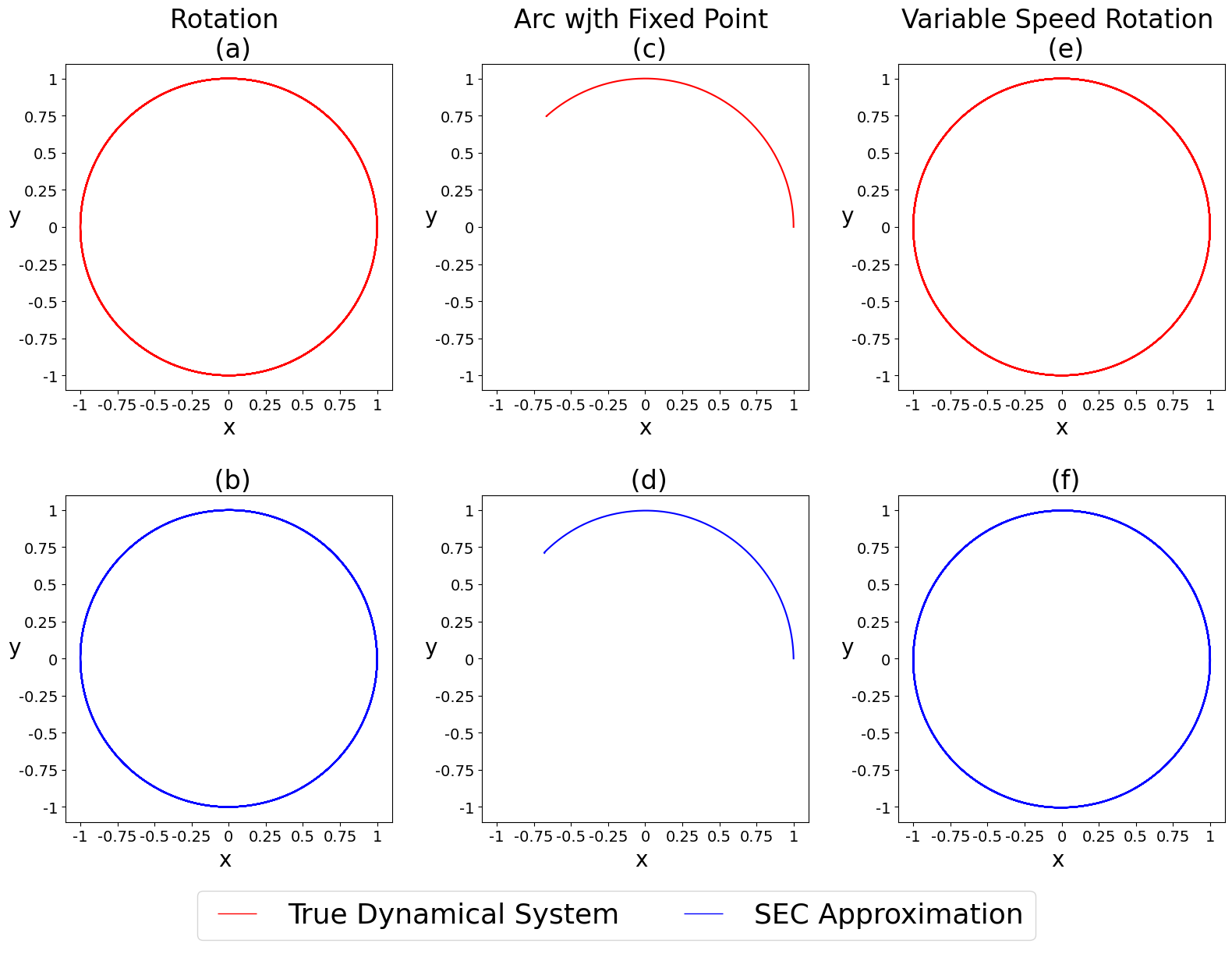}
  \label{fig:odesols1}
  \caption{Dynamical trajectories in $\mathbb R^2$ under the true ($y(t)$; a, c, e) and SEC-approximated vector fields ($y^{(\ell,\alpha)}$; b, d, f) of our circle examples, $V_1$ (a, b), $V_2$ (c,d), and $V_3$ (e, f). All trajectories evolve in a counterclockwise direction.}
\end{figure}

\begin{figure}
  \centering
  \includegraphics[width=10cm]{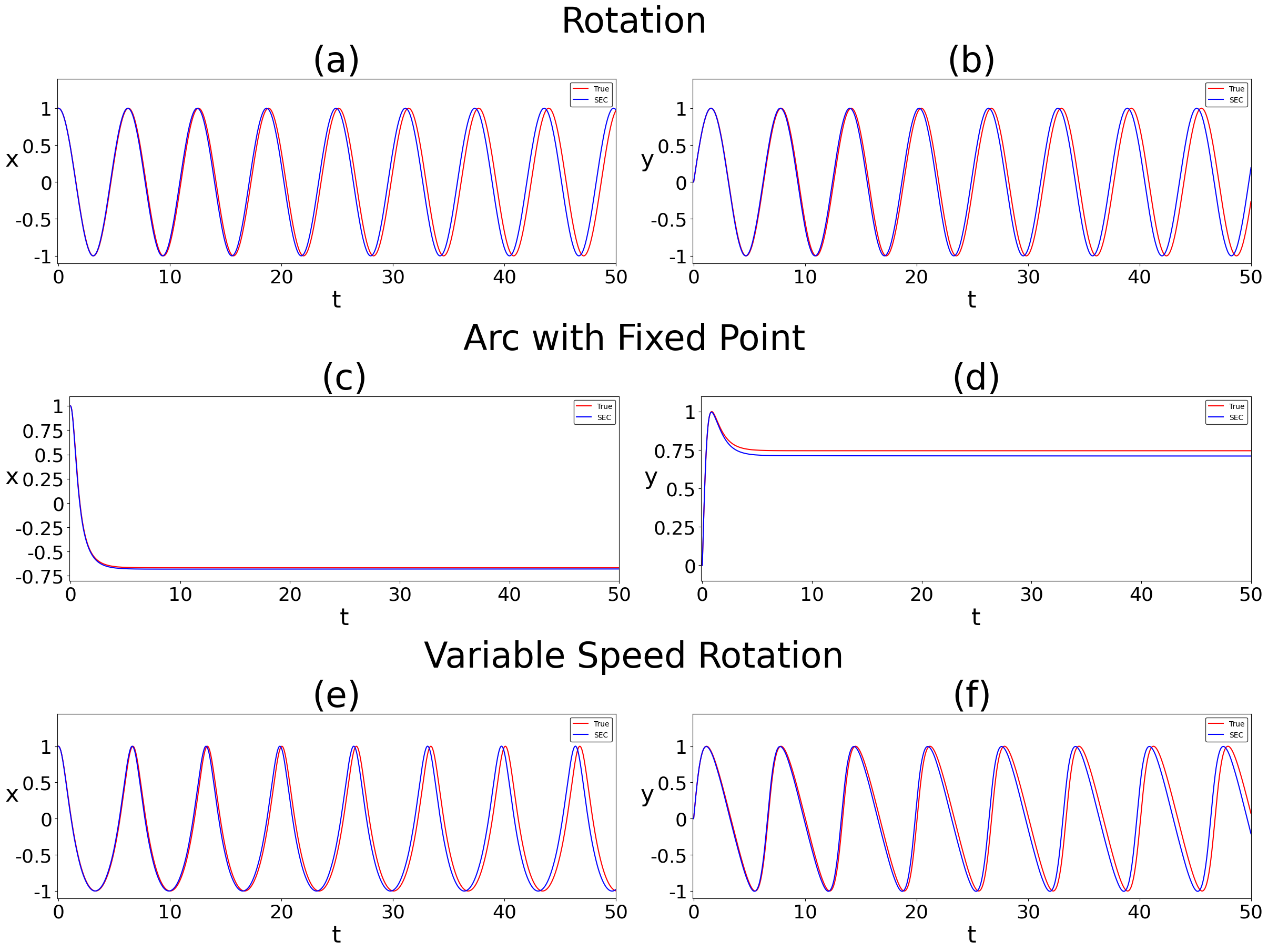}
  \label{fig:odetss1}
  \caption{As in \cref{fig:odesols1}, but showing time series of the $x$ and $y$ components of the $y(t)$ and $y^{(\ell,\alpha)}(t)$ trajectories.}
\end{figure}

\subsection{Dynamical systems on the 2-torus}
\label{subsec:torusroatation}
We consider the torus $\mathbb{T}^2$ embedded in $\real^3$ via the embedding $F: [0, 2\pi) \times [0, 2\pi) \rightarrow \real^3$,
\begin{equation}
    \label{eqn:torusembedding}
    F(\theta_1, \theta_2) = \begin{pmatrix} (a+b\cos{\theta_2})\cos{\theta_1} \\  (a+b\cos{\theta_2})\sin{\theta_1} \\ b\sin{\theta_2}.
    \end{pmatrix},
\end{equation}
Here, $a>0$ is the radius of the latitude circle with angle coordinate $\theta_1$ and $b >0$ is the radius of the meridian circle with coordinate angle $\theta_2$. We set $a = 5/3$ and $b = 3/5$ in all of our numerical experiments. Note that, unlike the circle embedding from \eqref{eqn:circleembedding} the components of the torus embedding \eqref{eqn:torusembedding} are not finite linear combinations of Laplace--Beltrami eigenfunctions---this is due to non-translation-invariance of the induced Riemannian metric $\mathfrak g$ from the torus embedding in $\mathbb R^3$. Moreover, $\mathfrak g$ has nonzero curvature in this example.

Analogously to~\eqref{eqn:s1vecdef}, every continuous vector field on $\mathbb{T}^2$ can be expressed as
\begin{equation}
    \label{eqn:t2vecdef}
    V = h_1(\theta_1, \theta_2)\frac{\partial}{\partial \theta_1} + h_2(\theta_1, \theta_2)\frac{\partial}{\partial \theta_2}
\end{equation}
for continuous functions $h_1$ and $h_2$. The corresponding pushforward vector field under $F$ from~\eqref{eqn:torusembedding} is given by
\begin{equation}
    \label{eqn:t2vF}
    \vec{V}(y) = (F_*V)(y) = \begin{pmatrix} -h_1(\theta_1, \theta_2)(a+b\cos{\theta_2})\sin{\theta_1} - h_2(\theta_1, \theta_2) (a+b\sin{\theta_2})\cos{\theta_1} \\  h_1(\theta_1, \theta_2)(a+b\cos{\theta_2})\cos{\theta_1} - h_2(\theta_1, \theta_2)(a+b\sin{\theta_2})\sin{\theta_1} \\
    h_2(\theta_1, \theta_2)b\cos{\theta_2}
    \end{pmatrix},
\end{equation}
where $y = F(\theta_1, \theta_2)$. In what follows, we consider three vector fields in this family:
\begin{itemize}
    \item A rational rotation with $h_1(\theta_1, \theta_2) = h_2(\theta_1, \theta_2) = 1$.
    \item An irrational rotation with $h_1(\theta_1, \theta_2) = 1$ and $h_2(\theta_1,\theta_2) = \mathbb \alpha \in \mathbb R \setminus \mathbb Q$.
    \item A Stepanoff flow \cite{Oxtoby53} with
        \begin{align*}
            h_1(\theta_1, \theta_2) &= h_2(\theta_1, \theta_2)+(1-\alpha)(1-\cos\theta_2),\\
            h_2(\theta_1, \theta_2) &= \alpha(1-\cos(\theta_1 - \theta_2)),
        \end{align*}
        for $\alpha \in \mathbb R \setminus Q$.
\end{itemize}
We set $\alpha = \sqrt{20}$ in both of the irrational rotation and Stepanoff flow examples. Similarly to the circle examples from \cref{subsec:circlerotation}, we denote the corresponding vector fields as $V_1$, $V_2$, and $V_3$, respectively. \Cref{fig:truevecfieldt2} visualizes these vector fields through quiver plots on the periodic domain $[0, 2\pi)^2$.

\begin{figure}
  \centering
  \includegraphics[width=12cm]{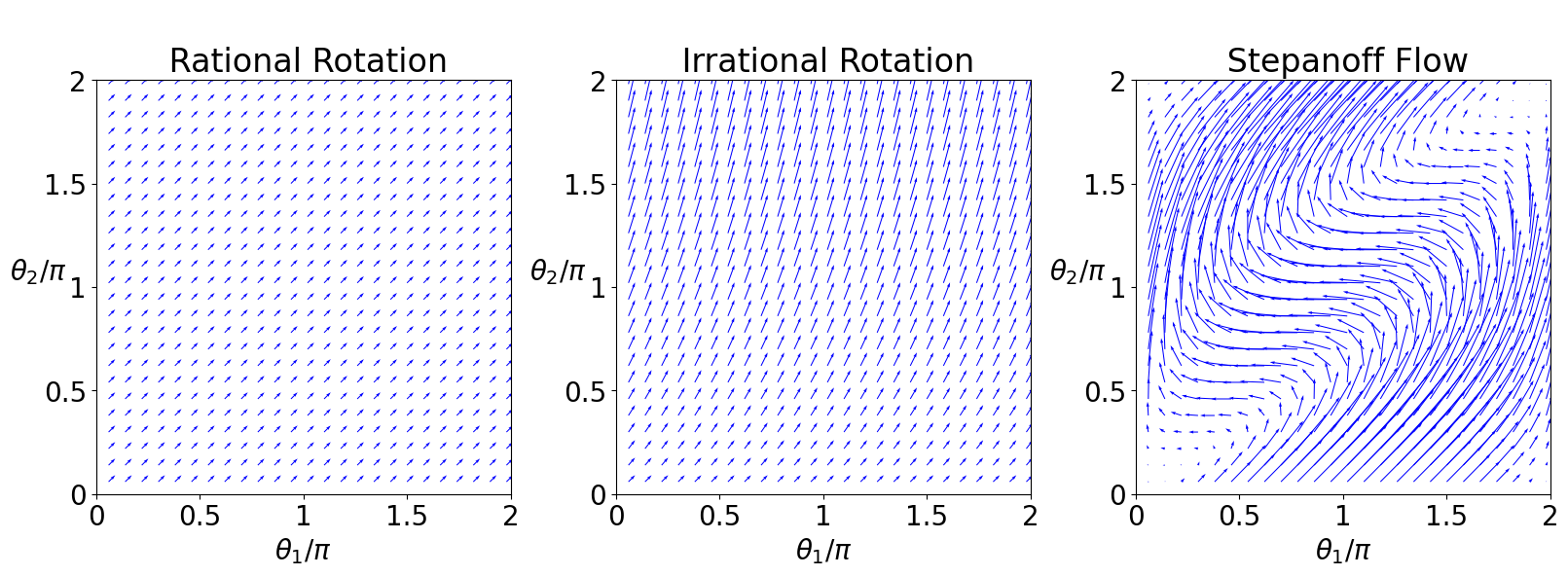}
  \label{fig:truevecfieldt2}
  \caption{Quiver plots of the vector fields $V_1$ (rational rotation with frequencies $(1,1)$), $V_2$ (irrational rotation with frequencies $(1, \sqrt{20})$), and $V_3$ (Stepanoff flow with parameter $\alpha = \sqrt{20}$). Arrows represent the components of the vector fields with respect to the coordinate basis fields $\frac{\partial\ }{\partial\theta_1}$ and $\frac{\partial\ }{\partial\theta_2}$ associated with the latitude ($\theta_1$) and meridian ($\theta_2$) angles on $\mathbb T^2$.}
\end{figure}

All of $V_1$, $V_2$, and $V_3$ have 0 divergence with respect to the Haar measure on $\mathbb T^2$ (i.e., $\partial_1 h_1 + \partial_2 h_2 = 0$), which implies that the corresponding flows are measure-preserving for the Haar measure. However, the divergence of these vector fields is nonzero with respect to the Riemannian volume measure $\nu$ induced from $F$, meaning that $\nu$ is not preserved by the dynamics. It should also be noted that the Stepanoff vector field $V_3$ has a fixed point at $(0,0)$ (see \cref{fig:truevecfieldt2}), making the corresponding flow weak-mixing in a topological sense. In \cite{Oxtoby53} it is shown that the Stepanoff flow is topologically conjugate to a time-reparameterized irrational rotation with basic frequencies $(1,\alpha)$, for a singular time-reparameterization function at $(0,0)$. Since we use the same value $\alpha = \sqrt{20}$ in both of the irrational rotation and Stepanoff flow examples, the two systems are related in this sense.

Our training dataset for each torus example consists of $N = 150 \times 150 = \text{22,500}$ points sampled on a uniform grid in $\mathcal M$ with respect to the angle coordinates $(\theta_1, \theta_2)$. Note that the sampling is non-uniform with respect to Riemannian measure but this does not affect the asymptotic convergence of the diffusion maps algorithm (see \cref{app:dm}). Throughout, we use the kernel bandwidth parameter $\epsilon = 0.0966$ and the SEC spectral resolution parameters $L = L_1 = 120$ and $\LD = L_2 = 280$. For these parameter choices the $R^2$ coefficients for the SEC approximations $\vec V^{(\ell,\alpha)}$ of $V_1$, $V_2$, and $V_3$ are 0.999975, 0.999945, and 0.999855, respectively, indicating high approximation accuracy with respect to the Hodge norm as in the circle examples. The reconstructed vector fields $\vec V^{(\ell,\alpha)}$ are plotted together with the true pushforwards $\vec V$ in the embedding space $\mathbb R^3$ in \cref{fig:tangvect2}, where the high approximation accuracy is visually evident.

\begin{figure}
  \centering
  \includegraphics[width=12cm]{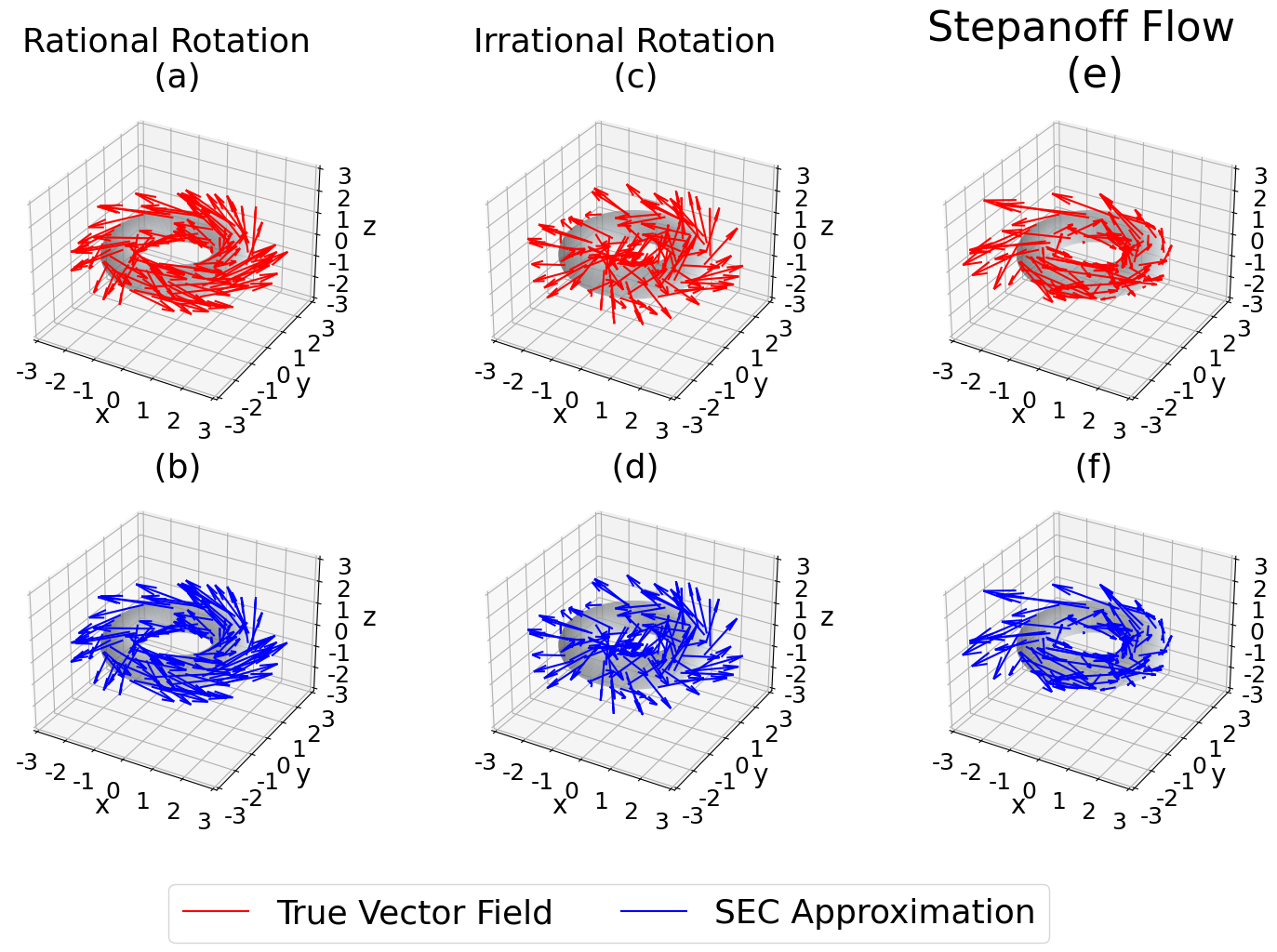}
  \label{fig:tangvect2}
    \caption{Quiver plots of the true ($\vec V$; a, c, e) and SEC-approximated vector fields ($\vec V^{(\ell,\alpha)}$; b, d, f) on the 2-torus embedded in $\mathbb R^3$. (a, b) Rational rotation, $V_1$; (c, d) irrational rotation, $V_2$; and (e, f) Stepanoff flow, $V_3$. The embedded torus $Y = F(\mathbb T^2)$ is shown in each panel as a shaded surface for reference. Arrows show tangent vectors on $Y$ for a subset of points in the training dataset.}
\end{figure}

In \cref{fig:odesolt2,fig:odetst2} we compare traceplots in $\mathbb R^3$ and the corresponding coordinate time series, respectively, of true dynamical trajectories, $y(t)$, generated by the vector fields $V_1$, $V_2$, and $V_3$, with the corresponding SEC approximations $y^{(\ell,\alpha)}(t)$. These trajectories are computed using similar numerical methods as in the circle examples (see \cref{subsec:circlerotation}). As initial conditions, we use $(\theta_{1,0},\theta_{2,0}) = (0, 0)$ for the rational and irrational rotations. The Stepanoff flow trajectories are initialized at $(\theta_{1,0},\theta_{2,0}) = (\pi+0.3, \pi+0.5)$ to avoid the trivial solution starting at the fixed point. The results demonstrate accurate trajectory reconstruction for the rational and irrational rotation systems over the entire integration intervals. In the case of the Stepanoff flow, while $y^{(\ell,\alpha)}(t)$ accurately tracks the true trajectory $y(t)$ for $t \lesssim 1.5$, at later times we see significant discrepancy buildup, leading to an apparent decorrelation of the true trajectories. While this behavior illustrates the challenges of orbit reconstruction in spite of high vector field reconstruction accuracy in the Hodge norm (see \cref{subsec:orbitrecon}), we should note that trajectory deviation does not necessarily indicate failure of the SEC in this system that exhibits sensitive dependence on initial conditions. In particular, the fact that $y^{(\ell, \alpha)}$ remains tangent to the embedded torus in $\mathbb R^3$ (see \cref{fig:odesolt2}), suggests that the SEC consistently represents this torus as an invariant set, despite nonlinearity and curvature of the embedding $F$.

\begin{figure}
  \centering
  \includegraphics[width=12cm]{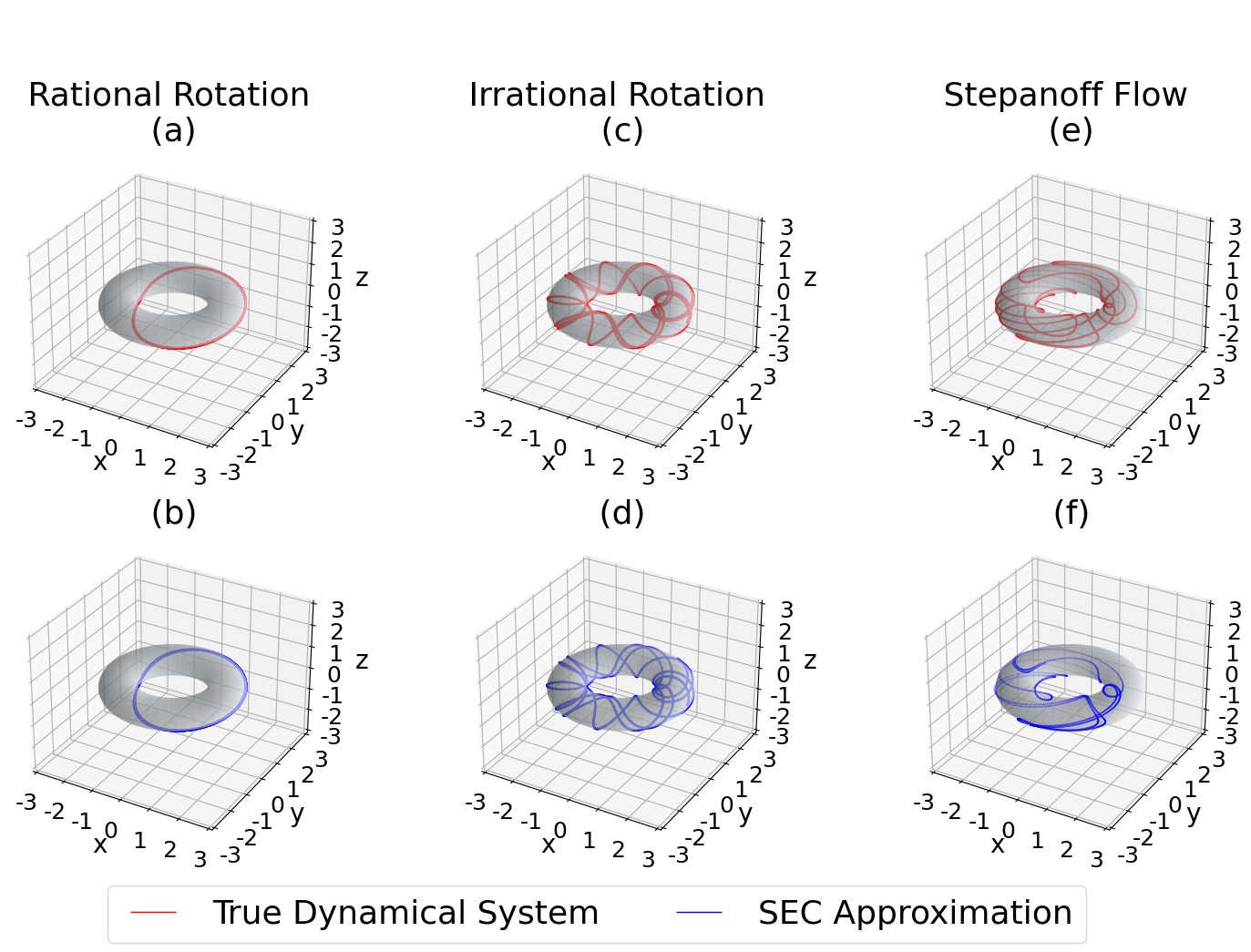}
  \label{fig:odesolt2}
  \caption{Dynamical trajectories in $\mathbb R^3$ under the true ($y(t); $a, c, e) and SEC-approximated vector fields ($y^{(\ell,\alpha)}(t); $b, d, f) of the torus examples $V_1$ (a, b), $V_2$ (c,d), and $V_3$ (e, f).}
\end{figure}

\begin{figure}
  \centering
  \includegraphics[width=10cm]{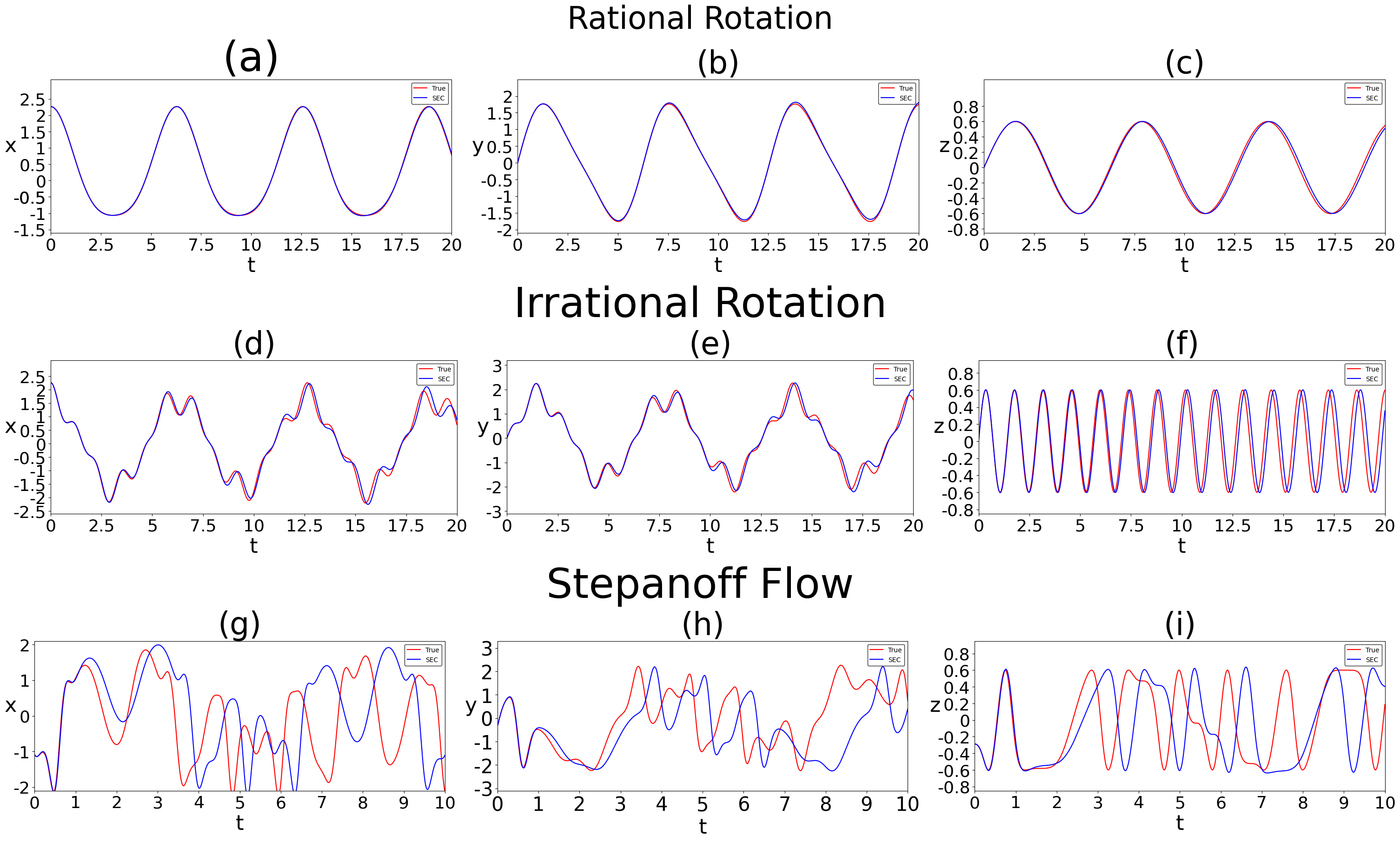}
  \label{fig:odetst2}
  \caption{As in \cref{fig:odesols1}, but showing time series of the $x$, $y$, and $z$ components of the $y(t)$ and $y^{(\ell,\alpha)}(t)$ trajectories.}
\end{figure}

\section{Discussion}
\label{sec:discussion}

In this paper, we have presented a geometrical technique for data-driven approximation of dynamical systems on Riemannian manifolds. Central to our approach is a formulation of supervised learning of dynamical systems as an approximation problem in a hypothesis space of vector fields (generators of dynamics), constructed from data sampled on the manifold using the spectral exterior calculus (SEC) \cite{BerryGiannakis20} technique. The SEC builds frames for $L^2$ spaces of vector fields consisting of products of eigenfunctions of the Laplace--Beltrami operator and their gradients, estimated using kernel techniques along with a carr\'e du champ identity \eqref{eqn:carre_du_champ} for gradient computation. The learned dynamical vector field, $V^{(\ell,\alpha)}$, is determined by solving a least-squares regression problem in the SEC frame using tangent vectors sampled on the manifold, $\mathcal M$, as training data, and can be represented intrinsically as a linear operator acting on functions by directional derivatives. This operator-theoretic representation allows reconstructing $V^{(\ell,\alpha)}$ as an embedded vector field $\vec V^{(\ell,\alpha)}$ (``arrow field'') in Euclidean data space $\mathbb R^d$ through the pushforward map, $\vec V^{(\ell,\alpha)} = V^{(\ell,\alpha)} F$, associated with the embedding $F$ of the manifold. The reconstructed vector field $\vec V^{(\ell,\alpha)}$ can then be used for a variety of tasks, including approximation of dynamical trajectories (initial-value prediction) which was studied in this paper.

Our analytical results include a proof of convergence of the SEC approximation $\vec V^{(\ell,\alpha)}$ in a Hodge norm on vector fields (\cref{thm:1}), as well as a bound on trajectory reconstruction error averaged over initial conditions on the manifold (\cref{prop:L2_vector_approx}). These results were obtained under general regularity assumptions on the dynamical system, sampling distribution of the training data, and spectral accuracy of the kernel-approximated Laplacian (\crefrange{A:1}{A:4}), without requiring foreknowledge of the dynamical system and/or manifold in parametric form.

A novelty of our technique over the conventional approach of representing vector fields componentwise in scalar-valued function spaces is that learning takes place in an intrinsic space of vector fields that can generate arbitrary (nonlinear) flows on the manifold in a manner compatible with its (nonlinear) embedding in data space. In particular, any loss of tangency of $\vec V^{(\ell,\alpha)}$ to the embedded image $F(\mathcal M) \subset \mathbb R^d$ occurs only in the final step of our reconstruction procedure due to finite-rank operator projection, avoiding spurious results due to learning in a hypothesis space containing $\mathbb R^d$-valued functions that are not tangent to $F(\mathcal M)$. We believe that the carefully chosen hypothesis space, leveraging the SEC and kernel methods, aids the consistency and stability of the learned dynamical system with our approach. In this paper, we presented a suite of numerical experiments involving dynamical systems on the circle and 2-torus that provide encouraging results on vector field reconstruction and approximation of dynamical orbits.

There are several avenues of further investigation stemming from this work. First, besides the local structure of the vector field, a dynamical system exhibits several global properties such as attractors, invariant measures, periodic orbits, and Lyapunov spectra. An interesting topic for future work would be to characterize the extent at which such properties can be reconstructed experimentally with the SEC-based approach. For example, co-existing periodic orbits reflect the complexity of the dynamics. A close match in the periodic orbits up to a certain period of the true and reconstructed systems indicate a good reconstruction. In such studies, a challenge that one invariably encounters is that global properties of dynamics can be extremely sensitive to and sometimes change discontinuously with changes in the vector field (e.g., \cite{ChekrounEtAl13,BerryDas2023learning, BerryDas2024review}), so the focus would be on parameter regimes with sufficient structural stability for data-driven analysis.

Second, the bounds in \Cref{prop:L2_vector_approx} can be potentially improved by working in spaces of vector fields equipped with stronger norm than the Hodge norm of $\mathbb H$. A good candidate is a combined Hodge--Sobolev norm that admits a weighted SEC frame construction \cite{BerryGiannakis20}. Recent work \cite{ParkEtAl24} has demonstrated accurate reconstruction of Lyapunov spectra with neural network models trained with a loss function that incorporates the Jacobian of the dynamics. It is possible that a Hodge--Sobolev SEC frame is similarly beneficial to statistical reconstruction accuracy by augmenting the pointwise comparison in the current formulation by a $C^1$ comparison.

Finally, it would be useful to explore ways of reducing computational cost. In our current implementation, both training and prediction phases are carried out using brute-force linear algebra methods (see \cref{fig:Time_complexity}). For a training dataset of size $N$ these methods materialize rank-$N$ kernel matrices with $O(N^2)$ and $O(N)$ time/space cost during training and prediction, respectively. Low-rank approximation methods, such as random feature projection and subsampling \cite{LongFerguson19,GiannakisEtAl23,EpperlyEtAl25}, provide promising routes for improving efficiency, particularly in timestepping problems requiring large numbers of sequential evaluations of $V^{(\ell,\alpha)}$.

\appendix

\section{Algorithms}
\label[app]{app:algorithms}

\subsection{Diffusion maps}
\label[app]{app:dm}
We give an outline of the diffusion maps algorithm \cite{CoifmanLafon2006} as used in the numerical experiments of \cref{sec:numexp}.

Our starting point is a Gaussian radial basis function (RBF) kernel $k_\epsilon \colon \mathbb R^d \times \mathbb R^d \to \mathbb R_{>0}$ on data space,
\begin{equation}
    \label{eqn:k_gauss}
    k_\epsilon(y, y') = \exp\left(- \frac{\lVert y - y'\rVert_2^2}{\epsilon^2}\right),
\end{equation}
parameterized by a bandwidth (lengthscale) parameter $\epsilon>0$. Let $\alpha = (\epsilon,N)$ as in \cref{sec:data_driven}. Associated with $k_\epsilon$ are right and left normalization functions $\deg_r^{(\alpha)},\deg_l^{(\alpha)} \in C^\infty(\mathbb R^d)$, where
\begin{equation}
    \label{eqn:deg}
    \begin{aligned}
        \deg_r^{(\alpha)}(y) &= \int_{\mathcal M} k_\epsilon(y, F(x))\, d\mu_N(x) = \frac{1}{N} \sum_{n=1}^N k_\epsilon(y, y_n), \\
        \deg_l^{(\alpha)}(y) &= \int_{\mathcal M} \frac{k_\epsilon(y, F(x))}{\deg_r^{(\alpha)}(x)}\, d\mu_N(x) = \frac{1}{N}\sum_{n=1}^N \frac{k_\epsilon(y, y_n)}{\deg_r^{(\alpha)}(y_n)}.
    \end{aligned}
\end{equation}
By construction, both normalization functions are strictly positive with images bounded away from zero on compact sets. Defining $\rho_\alpha \colon \mathbb R^d \times \mathbb R^d \to \mathbb R_{>0}$ as
\begin{displaymath}
    \rho_\alpha(y, y') = \frac{k_\epsilon(y, y')}{\deg_l^{(\alpha)}(y) \deg^{(\alpha)}_r(y')},
\end{displaymath}
we then have that this kernel is $C^\infty$ and Markovian in the sense of~\cref{eqn:markov}, and induces a Markov operator $P_{\alpha} \colon L^2(\mu_N) \to L^2(\mu_N)$. Because $p_\alpha(x, x') >0$, this operator is ergodic, i.e., $P_\alpha f = f$ iff $f(x_n)$ is constant for all $ n \in \{ 1, \ldots, N \}$.

Let $p_\alpha\colon \mathcal M \times \mathcal M \to \mathbb R$ be the pullback kernel on the manifold, $p_\alpha(x, y) = \rho_\alpha(F(x), F(y))$ as in \cref{subsec:kernel}, and define $d_\alpha \in C^\infty(\mathcal M)$ as
\begin{displaymath}
    d_\alpha(x) = \sqrt{\deg_l^{(\alpha)}(F(x)) / \deg_r^{(\alpha)}(F(x))}.
\end{displaymath}
One then verifies that $\tilde p_\alpha \colon \mathcal M \times \mathcal M \to \mathbb R_{>0}$ with
\begin{displaymath}
    \tilde p_\alpha(x, x') = d_\alpha(x) p_\alpha(x, x')/ d_\alpha(x')
\end{displaymath}
is a symmetric kernel, and thus induces a self-adjoint integral operator $\tilde P_\alpha \colon L^2(\mu_N) \to L^2(\mu_N)$, where
\begin{displaymath}
    \tilde P_\alpha f = \int_{\mathcal M} \tilde p_\alpha(\cdot, x)f(x)\, d\mu_N(x) \equiv \frac{1}{N}\sum_{n=1}^N \tilde p_\alpha(\cdot, x_n)f(x_n).
\end{displaymath}
This operator is related to $P_\alpha$ by the similarity transformation
\begin{equation}
    \label{eqn:p_similarity}
    \tilde P_\alpha = D_\alpha P_\alpha D_\alpha^{-1},
\end{equation}
where $D_\alpha \colon L^2(\mu_N) \to L^2(\mu_N)$ is the multiplication operator by $d_\alpha$. Computationally, $D_\alpha$ is represented by the $N\times N$ diagonal matrix $\bm D_\alpha = [D_{ij}^{(\alpha)}]$ with diagonal entries $D_{ii}^{(\alpha)} = d_\alpha(x_i)$. Moreover, we have $\bm{\tilde P}_\alpha = \bm D_\alpha \bm P_\alpha \bm D_\alpha^{-1}$, where $\bm{\tilde P}_\alpha = [\tilde P^{(\alpha)}_{ij}]$ and $\bm P_\alpha = [P^{(\alpha)}_{ij}]$ are the matrix representations of $\tilde P_\alpha$ and $ P_\alpha$, respectively, with $\tilde P_{ij}^{(\alpha)} = \tilde p_\alpha(x_i, x_j)/N$ and $P_{ij}^{(\alpha)} = p_\alpha(x_i, x_j) / N$.

By~\cref{eqn:p_similarity}, $P_\alpha$ and $\tilde P_\alpha$ have the same eigenvalues $\Lambda^{(\alpha)}_j$. The eigenvalues are therefore real by self-adjointness of $\tilde P_\alpha$, and satisfy $\lvert \Lambda_j^{(\alpha)}\rvert \leq 1$ with equality for a simple eigenvalue equal to 1 by Markovianity and ergodicity of $P_\alpha$. It can also be shown that $\Lambda^{(\alpha)}_j >0$ since the Gaussian kernel $k_\epsilon$ is strictly positive-definite. We can thus conclude that the eigenvalues admit the ordering $1 = \Lambda_0^{(\alpha)} > \Lambda_1^{(\alpha)} \geq \Lambda_2^{{(\alpha)}} \geq \cdots \geq \Lambda_{N-1}^{(\alpha)} > 0$.

Next, consider an eigendecomposition of $\tilde P_\alpha$,
\begin{displaymath}
    \tilde P_\alpha e_j^{(\alpha)} = \Lambda_j^{(\alpha)} e_j^{(\alpha)},
\end{displaymath}
where the eigenvectors $e_j^{(\alpha)}$ are chosen to be orthogonal in $L^2(\mu_N)$. By the similarity relation~\cref{eqn:p_similarity}, it follows that
\begin{displaymath}
    \phi_j^{(\alpha)} = D^{-1}_\alpha e_j^{(\alpha)} = e_j^{(\alpha)} / d_\alpha
\end{displaymath}
are eigenvectors of $P_\alpha$ corresponding to $\Lambda_j^{(\alpha)}$ and $\{ \phi_j^{(\alpha)} \}_{j=0}^{N-1}$ forms a basis of $L^2(\mu_N)$. As a result, the data-driven Laplacian~\eqref{eqn:lapl_epsilon} is defined uniquely from the eigendecomposition
\begin{displaymath}
    \Delta_{\alpha}\phi_j^{(\alpha)} = \lambda_j^{(\alpha)} \phi_j^{(\alpha)}, \quad \lambda_j^{(\alpha)} = - \frac{\log \Lambda_j^{(\alpha)}}{c\epsilon^2}.
\end{displaymath}
For the Gaussian RBF kernel~\eqref{eqn:k_gauss} the choice $c=1/4$, which we adopt throughout, leads to a consistent approximation of the Laplace--Beltrami operator $\Delta$ by $\Delta_\alpha$ \cite{CoifmanLafon2006}.

Let now $w_\alpha \in L^2(\mu_N)$ be a top eigenvector of $P^*_\alpha$, i.e., $P^*_\alpha w_\alpha = w_\alpha$. Every such eigenvector satisfies $w_\alpha(x_n) = \left(e_0^{(\alpha)}(x_n)\right)^2 / Z_\alpha$ for a normalization constant $Z_\alpha$. To fix this constant, suppose that $Z_\alpha > 0$, and observe that the eigenvectors $\phi_j^{(\alpha)}$ are orthogonal with respect to the inner product
\begin{equation}
    \label{eqn:phi_innerprod_alpha}
    \left\langle \phi_i^{(\alpha)}, \phi_j^{(\alpha)}\right\rangle_{L^2(\nu_\alpha)} := \int_{\mathcal M} \phi_i^{(\alpha)} \phi_j^{(\alpha)} \, d\nu_\alpha \equiv \int_{\mathcal M} \phi_i^{(\alpha)} \phi_j^{(\alpha)} w_\alpha \, d\mu_N
\end{equation}
associated with the weighted sampling measure $\nu_\alpha = \sum_{n=1}^N w_\alpha(x_n) \delta_{x_n}$. Similarly, we can express inner products of the Laplace--Beltrami eigenfunctions $\phi_j$ as
\begin{equation}
    \label{eqn:phi_innerprod}
    \langle \phi_i, \phi_j\rangle_{L^2(\nu)} := \int_{\mathcal M} \phi_i\phi_j \, d\nu = \int_{\mathcal M}\phi_i\phi_j\frac{1}{\sigma}\, d\mu,
\end{equation}
where $\frac{1}{\sigma} \in C^\infty(\mathcal M)$ is the reciprocal of the sampling density $\sigma$ (see \cref{A:3}). By convergence of $P_\alpha$ and $P_\alpha^*$ to the heat operator $e^{-c\epsilon^2 \Delta}$ (implied by \cref{A:4}), it follows that $w_\alpha$ approximates $\frac{1}{\sigma}$ up to normalization. Given, moreover, that $\int_{\mathcal M} \frac{1}{\sigma} \, d\mu = \nu(\mathcal M)$, we can ensure asymptotic consistency between the $L^2(\nu_\alpha)$ and $L^2(\nu)$ inner products by choosing the normalization constant $Z_\alpha$ such that
\begin{equation}
    \label{eqn:weight_normalization}
    \int_{\mathcal M} w_\alpha \, d\mu_N = \mathcal V_\alpha,
\end{equation}
for an estimate $\mathcal V_\alpha$ of the Riemannian volume $\nu(\mathcal M)$ that converges in the large-data limit $\alpha \to (0^+, \infty)$.

To obtain such a volume estimate $\mathcal V_\alpha$ we use the Minakshisundaram--Pleijel formula for the heat trace \cite{MinakshisundaramPleijel49},
\begin{displaymath}
    \int_{\mathcal M} h_\tau(x, x) \, d\nu(x) \simeq (4 \pi \tau)^{-\dim(\mathcal M)/2} \nu(\mathcal M), \quad \tau \ll 1,
\end{displaymath}
where $h_\tau \colon \mathcal M \times \mathcal M \to \mathbb R_+$ is the time-$\tau$ heat kernel associated with the Riemannian metric $\mathfrak g$ (that is; $e^{-\tau \Delta} f = \int_{\mathcal M}h_\tau(\cdot, x) f(x) \, d\nu(x)$ for $\tau>0$ and $f \in L^2(\nu)$). In our setting, $P_\alpha$ with $\alpha = (\epsilon, N)$ approximates the heat operator at time $\tau = c\epsilon^2 = \epsilon^2/4$, and thus $\int_{\mathcal M} p_\alpha(x, x) \, d\mu_N(x)$ approximates the corresponding heat trace $\int_X h_\tau(x, x) \, d\nu(x)$. This leads to the following asymptotic estimate for $\nu(\mathcal M)$,
\begin{equation}
    \label{eqn:vol_est}
    \nu(\mathcal M) \simeq \mathcal V_\alpha = (\pi \epsilon^2)^{\dim(\mathcal M)/2} \int_{\mathcal M} p_\alpha(x, x) \, d\mu_N(x) = (\pi \epsilon^2)^{\dim(\mathcal M)/2} \tr \bm P_\alpha.
\end{equation}
Note that the formula above requires knowledge of the manifold dimension $\dim(\mathcal M)$, or a sufficiently accurate estimate thereof. In the examples of \cref{sec:numexp}, $\dim(\mathcal M)$ is assumed to be known. Numerical methods for estimating $\dim(\mathcal M)$ can be found, e.g., in \cite{LittleEtAl11,BerryEtAl15,CoifmanEtAl08}.

Having computed $\mathcal V_\alpha$, we set the normalization coefficient $Z_\alpha$ and weight $w_\alpha$ so as to satisfy~\eqref{eqn:weight_normalization}. We then normalize our diffusion eigenvectors such that they are orthonormal in $L^2(\nu_\alpha)$, $\langle \phi_i^{(\alpha)}, \phi_j^{(\alpha)}\rangle_{L^2(\nu_\alpha)} = \delta_{ij}$.

\subsection{Pseudocode}
This appendix includes pseudocode for computation of the diffusion maps eigenpairs (\cref{algo:dffsn}), out-of-sample evaluation of the diffusion eigenvectors (\cref{algo:eigs}), and vector field regression problem in the SEC frame (\cref{algo:bcoeff}). In the pseudocode, $\bm u^{\odot 2}$ and $\bm u./ \bm v$ denote elementwise squaring and division of vectors $\bm u, \bm v \in \mathbb R^N$, respectively. Moreover, $\langle \bm u, \bm v\rangle_{\bm w} = \bm u^\top (\diag \bm w) \bm v$ is the weighted Euclidean inner product induced by a vector $\bm w \in \mathbb R^N$ with strictly positive elements, and $\lVert \bm v \rVert_{\bm w} = \sqrt{\langle \bm v, \bm v\rangle_{\bm w}}$ the corresponding norm. The sampled points $ \{ y_n \in \mathbb R^d \}_{n=1}^N$ on the embedded manifold $F(\mathcal M)$ and corresponding samples $ \{ \vec v_n \}_{n=1}^N$ of the pushforward vector field $F_* V$ are as per the description in \cref{sec:probstatement}.

\begin{algorithm}
    \caption{Eigenpairs $(\Lambda_j^{(\alpha)}, \bm \phi_j^{(\alpha)})$ of the heat operator matrix $\bm P_\alpha$ from \cref{app:dm}.}
    \label{algo:dffsn}
    \begin{itemize}[wide]
        \item \textbf{Inputs}: Dataset $\{y_n\in\real^d\}_{n=1}^N$; kernel bandwidth $\epsilon>0$; manifold dimension $m \in \mathbb N$; number of eigenpairs $M \in \mathbb N$.
        \item \textbf{Outputs}: Inner product weights $\bm w_\alpha \in \mathbb R^N$; top eigenvalues $\{\Lambda_j^{(\alpha)}\}_{j=0}^{M-1}$, sorted in descending order; corresponding eigenvectors $ \{ \bm \phi_j^{(\alpha)} \}_{j=0}^{M-1}$, orthonormalized with respect to $\bm w_\alpha$; corresponding Laplace--Beltrami eigenvalues $ \{ \lambda_j^{(\alpha)} \}_{j=0}^{M-1}$.
        \item \textbf{Steps}: With $k_\epsilon$ from~\eqref{eqn:k_gauss} as the kernel function and $\deg_r^{(\alpha)}, \deg_l^{(\alpha)}$ from~\eqref{eqn:deg} as the degree (normalization) functions:
        \begin{enumerate}
            \item Compute the values $r_n = \deg_r^{(\alpha)}(y_n)$ and $l_n = \deg_l^{(\alpha)}(y_n)$ of the degree functions for $n \in \{1, \ldots, N \}$.
            \item Form the $N\times N$ symmetric matrix $\bm{\tilde P}_\alpha = [\tilde P_{ij}^{(\alpha)}]_{i,j=1}^N$ with entries
                \begin{displaymath}
                    \tilde P_{ij}^{(\alpha)} = \frac{k_\epsilon(y_i, y_j)/N}{\sqrt{l_ir_ir_j l_j}}.
                \end{displaymath}
            \item Compute the largest $M$ eigenvalues $\Lambda_0^{(\alpha)}, \ldots, \Lambda_{M-1}^{(\alpha)} \in [0, 1]$ of $\bm{\tilde P}_\alpha$ and corresponding eigenvectors $\bm e_j \in \mathbb R^N$. Sort the $\Lambda_j^{(\alpha)}$ in descending order.
            \item Compute the volume estimate $\mathcal V_\alpha$ using~\eqref{eqn:vol_est} with $\dim(\mathcal M) = m$.
            \item Compute the weight vector $\bm w_\alpha = \mathcal V_\alpha \bm e_0^{\odot 2}/\lVert  \bm e_0^{\odot 2}\rVert_1$.
            \item For $j \in \{0, \ldots, M-1\}$, compute the eigenvectors
                \begin{displaymath}
                    \bm \phi_j^{(\alpha)} = \bm v_j/\lVert \bm v_j\rVert_{\bm w_\alpha}, \quad \bm v_j = \bm e_j ./ \bm e_0.
                \end{displaymath}
            \item For $j \in \{0, \ldots, M-1\}$, compute the Laplace--Beltrami eigenvalues $\lambda_j^{(\alpha)} = - \frac{4 \log \Lambda_j^{(\alpha)}}{\epsilon^2}$.
            \item Return $\bm w_\alpha$, $\{\Lambda_j^{(\alpha)}\}_{j=0}^{M-1}$, $ \{ \bm \phi_j^{(\alpha)} \}_{j=0}^{M-1}$, $ \{ \lambda_j^{(\alpha)} \}_{j=0}^{M-1}$.
        \end{enumerate}
    \end{itemize}
\end{algorithm}

\begin{algorithm}
    \caption{\color{black}Out-of-sample extension of diffusion eigenfunctions.}
    \label{algo:eigs}
    \begin{itemize}[wide]
        \item \textbf{Inputs}: Dataset $\{y_n\in\real^d\}_{n=1}^N$; kernel bandwidth $\epsilon>0$; eigenpairs $ \{ \Lambda_j^{(\alpha)}, \bm \phi_j^{(\alpha)} \}_{j=0}^{M-1}$ from \cref{algo:dffsn}.
        \item \textbf{Outputs:} Out-of-sample extensions $ \{ \varphi_j^{(\alpha)}\colon \mathbb R^d \to \mathbb R \}_{j=0}^{M-1}$ of the $\bm \phi_j$.
            \begin{enumerate}
                \item Build the vector-valued function $\bm \rho \colon \mathbb R^d \to \mathbb R^N$, where
                    \begin{displaymath}
                        \bm \rho(y) = \left(\frac{k_\epsilon(y, y_1)}{\deg_l^{(\alpha)}(y) \deg_r^{(\alpha)}(y_1)}, \ldots, \frac{k_\epsilon(y, y_N)}{\deg_l^{(\alpha)}(y) \deg_r^{(\alpha)}(y_N)} \right).
                    \end{displaymath}
                \item For $j \in \{ 0, \ldots, M-1 \}$, build the functions $\varphi_j^{(\alpha)} \colon \mathbb R^d \to \mathbb R$, where
                    \begin{displaymath}
                        \varphi_j^{(\alpha)}(y) = \frac{1}{N\Lambda_j^{(\alpha)}} \bm \rho(y) \cdot \bm \phi_j^{(\alpha)}.
                    \end{displaymath}
                \item Return $ \{ \varphi_j^{(\alpha)} \}_{j=0}^{M-1}$.
            \end{enumerate}
    \end{itemize}
\end{algorithm}

\begin{algorithm}
    \caption{Regression problem for the dynamical vector field $V$ in the SEC frame.}
    \label{algo:bcoeff}
    \begin{itemize}[wide]
        \item \textbf{Inputs:} Manifold samples $\{y_n \}_{n=1}^N$; vector field samples $ \{ \vec v_n \in \mathbb R^d \}_{n=1}^N$; number of gradient fields $J \in \mathbb N$ in the SEC frame; spectral resolution parameters $L, \LD, L_1, L_2 \in \mathbb N$; spectral truncation parameter $\eta>0$; eigenpairs $ \{ (\lambda_j^{(\alpha)}, \bm \phi_j^{(\alpha)}) \}_{j=0}^{M-1}$ and inner product weights $\bm w_\alpha$ from \cref{algo:dffsn} with $M = \max \{J, L, \LD, L_1, L_2\} $.
        \item \textbf{Outputs:} Frame expansion coefficients $ \{ b_{ij}^{(\ell,\alpha)} \in \mathbb R \}_{i,j = 0,1}^{L-1, J}$ for the vector field $\vec V^{(\ell,\alpha)}$.
        \item \textbf{Steps}:
        \begin{enumerate}
            \item Compute $c$-coefficients associated with the eigenvectors $\bm \phi_j^{(\alpha)}$,
                \begin{displaymath}
                    c_{ijp}^{(\alpha)} = \langle \bm \phi_p^{(\alpha)}, \bm \phi_i^{(\alpha)} \bm \phi_j^{(\alpha)}\rangle_{\bm w_\alpha}, \quad i,j,p \in \{ 1, \ldots, M \}.
                \end{displaymath}
            \item Using the $c_{ijp}^{(\alpha)}$ coefficients and the eigenvalues $\lambda_j^{(\alpha)}$, compute $d_{ijkl}^{(\LD,\alpha)}$ for $i \in \{ 0, \ldots, L-1 \}$, $j \in \{ 1,\ldots, J \}$, $k \in \{ 0, \ldots, L_1 -1 \}$, and $l \in \{ 0, \ldots, L_2 -1 \}$ via the formula~\eqref{eqn:approx:d_coeff}. Reshape the resulting tensor $[d_{ijkl}^{(\LD, \alpha)}]$ into an $(LJ) \times (L_1L_2)$ matrix $\bm D$ by grouping together the $(i,j)$ and $(k,l)$ indices.
            \item Form the $d \times N$ data matrix $\bm Y = (y_1 \cdots y_N)$ with columns $y_n$ and compute the leading $L_1$ expansion coefficients $F_k^{(\alpha)} \in \mathbb R^d$ of the embedding,
                \begin{displaymath}
                    F_k^{(\alpha)} = \langle F, \phi_k^{(\alpha)}\rangle_{L^2(\nu_\alpha)} = \bm Y (\diag \bm w_\alpha) \bm \phi_k, \quad k \in \{ 0, \ldots, L_1 -1 \}.
                \end{displaymath}
            \item Form the $d \times N$ vector field matrix $\bm V = (\vec v_1 \cdots \vec v_N)$ with columns $\vec v_n$ and compute the coefficients
                \begin{displaymath}
                    \hat v_{kl} = F_k^{(\alpha)} \cdot \bm V (\diag \bm w_\alpha) \bm \phi_l, \quad  k \in \{ 0, \ldots, L_1 -1 \}, \quad l \in \{ 0, \ldots, L_2 -1 \}.
                \end{displaymath}
                Reshape the resulting tensor $[\hat v_{kl}]$ into an $(L_1L_2)$-dimensional vector $\bm{\hat v}$.
            \item Using the $c_{ijp}^{(\alpha)}$ coefficients and the eigenvalues $\lambda_j^{(\alpha)}$, compute the matrix elements $d_{ijlk}^{(\LD,\alpha)}$  of the Gram operator $G^{(\LD,\alpha)}$ in \eqref{eqn:g_d_ops_data_driven} for $i,k \in \{0, \ldots, L_1 \}$ and $j,l \in \{ 1, \ldots, J \}$. Reshape the resulting tensor $[d_{ijlk}^{(\LD,\alpha)}]$ into an $(LJ) \times (LJ)$ matrix $\bm G$ by grouping together the $(i,j)$ and $(l,k)$ indices.
            \item Compute an eigendecomposition $\bm G \bm u_j = \gamma_j \bm u_j$ of $\bm G$, where the eigenvectors $\bm e_j$ are orthonormal in $\mathbb R^{LJ}$.
            \item Form the spectrally truncated Gram matrix $\bm G_\eta = \sum_{j: \gamma_j > \eta} \gamma_j \bm u_j \bm u_j^\top$.
            \item Compute the solution $\bm b \in \mathbb R^{LJ}$ of the regression problem~\eqref{eqn:b_regression_datadriven}, $\bm b = \bm G_\theta^+ \bm D \bm{\hat v}$.
            \item Return the coefficients $ \{ b_{ij}^{(\ell,\alpha)} \}_{i,j=0,1}^{L-1,J}$ from the elements of this vector.
        \end{enumerate}
    \end{itemize}
\end{algorithm}

\section{Proof of \cref{prop:L2_vector_approx}} \label[app]{app:proof_orbit_recon} We shall make use of two auxiliary results, \cref{lem:running_away,lem:prod_log}. In what follows, $W_{-1} \colon [1/e, 0) \to \mathbb R_{<0}$ is the $-1$ branch of the Lambert $W$ function (also known as product logarithm function), defined as the multivalued inverse of the function $w \mapsto we^w$; see e.g., \cite{CorlessEtAl96}.

\begin{lemma}[``Running-away inequality'']
    \label{lem:running_away}
    Let $E$ be a separable Banach space, $T > 0$, and $f \colon [0, T) \to E$ differentiable at every $t \in [0, T)$. Then $t \mapsto \lVert f(t)\rVert_E$ is differentiable for every $t \in [0, T)$ for which $f(t) \neq 0$, and it is right-differentiable everywhere. Moreover, the bound
    \begin{displaymath}
        \frac{d\ }{dt} \lVert f(t)\rVert_E \leq \left\lVert \frac{d f(t)}{dt}\right\rVert_E
    \end{displaymath}
    holds for every $t \in [0, T)$ with right-hand derivatives used when necessary.
\end{lemma}

\begin{proof}
    See \cite[Lemma~3.3.1]{Murdock99} for a proof when $E$ is finite-dimensional. The proof in the infinite-dimensional case proceeds similarly, using a version of Taylor's theorem for separable infinite-dimensional Banach spaces.
\end{proof}

\begin{lemma}
    \label{lem:prod_log}
    For every $\alpha,\beta>0$, the solution of the initial-value problem
    \begin{displaymath}
        \dot h(t) = \alpha + \beta \sqrt{h(t)}, \quad h(0) = 0,
    \end{displaymath}
    is
    \begin{displaymath}
        h(t) = \left( \frac{\alpha}{\beta} (W_{-1}(\zeta(t)) + 1) \right)^2, \quad \zeta(t) = - \exp\left(- \frac{\beta^2 t}{\alpha} - 1\right).
    \end{displaymath}
\end{lemma}

\begin{proof}
    The $W_{-1}$ function on $[-1/e, 0)$ can be equivalently defined as the solution of the initial-value problem \cite{CorlessEtAl96}
    \begin{displaymath}
        \dot W_{-1}(z) = \frac{W_{-1}(z)}{z(1 + W_{-1}(z))}, \quad W_{-1}(-1/e) = -1.
    \end{displaymath}
    The claim of the lemma follows by differentiating the given expression for $h(t)$ and using this result.
\end{proof}

Returning to the proof of the proposition, we use \cref{lem:running_away} to get
\begin{align*}
    \dot\varepsilon(t) &\leq
    \norm{ \frac{d}{dt} y(t, \cdot) - \frac{d}{dt} \hat{y}(t, \cdot) }_{L^1(\nu; \mathbb R^d)} = \norm{ V \paran{ F\circ \Phi^t } - \hat{V} \paran{ F\circ \hat{\Phi}^t } }_{L^1(\nu; \mathbb R^d)} \\
    &\leq \norm{ V \paran{ F\circ \Phi^t } - \hat{V} \paran{ F\circ \Phi^t } }_{L^1(\nu; \mathbb R^d)} + \norm{ \hat{V} \paran{ F\circ \Phi^t } - \hat{V} \paran{ F\circ \hat{\Phi}^t } }_{L^1(\nu; \mathbb R^d)}.
\end{align*}
Writing $F = (F_1, \ldots, F_d)$ where $F_j \in C^\infty(\mathcal M)$ are the component functions of the embedding, we can estimate the first term in the last line as
\begin{align*}
    \norm{ V \paran{ F\circ \Phi^t } - \hat{V} \paran{ F\circ \Phi^t } }_{L^1(\nu; \mathbb R^d)}
    &= \int_{\mathcal M} \left\lVert (V - \hat V)(F \circ \Phi^t)(x)\right\rVert_1\, d\nu(x) \\
    &= \int_{\mathcal M}\sum_{j=1}^d \left\lvert(V - \hat V) (F_j \circ \Phi^t)(x) \right\rvert \, d\nu(x) \\
    &= \int_{\mathcal M}\sum_{j=1}^d \left\lvert(V - \hat V) \cdot \nabla (F_j \circ \Phi^t)(x) \right\rvert \, d\nu(x) \\
    &\leq \int_{\mathcal M} \sum_{j=1}^d \left\lVert V - \hat V\right\rVert_{\mathfrak g(x)} \left\lVert \nabla(F_j \circ \Phi^t)\right\rVert_{\mathfrak g(x)}\, d\nu(x)\\
    &\leq \left\lVert V - \hat V\right\rVert_{\mathbb H} \sum_{j=1}^d \sqrt{\int_{\mathcal M} \left\lVert \nabla(F_j \circ \Phi^t)\right\rVert^2_{\mathfrak g(x)}\, d\nu(x)}\\
    &= \lVert V - \hat V \rVert_{\mathbb H} \sum_{j=1}^d \left\lVert \left\lVert \nabla (F_j \circ \Phi^t)\right\rVert_{\mathfrak g}\right\rVert_{L^2(\nu)}\\
    &\leq \lVert V - \hat V \rVert_{\mathbb H} \sum_{j=1}^d \left\lVert  \nabla (F_j \circ \Phi^t)\right\rVert_{\mathbb H}\\
    &= \lVert V - \hat V \rVert_{\mathbb H} \lvert F \circ \Phi^t\rvert_{1,2} \leq \alpha.
\end{align*}
For the second term, letting $\Xi_j = F_j \circ \Phi^t - F_j \circ \hat \Phi^t$ and proceeding as above, we get
\begin{align*}
    \norm{ \hat{V} \paran{ F\circ \Phi^t } - \hat{V} \paran{ F\circ \hat{\Phi}^t } }_{L^1(\nu; \mathbb R^d)}
    & \leq \left\lVert \hat V\right\rVert_{\mathbb H} \sum_{j=1}^d \left\lVert\left\lVert \nabla \left(F_j \circ \Phi^t - F_j \circ \hat \Phi^t\right)\right\rVert_{\mathfrak g}\right\rVert_{L^2(\nu)} \\
    & \leq \left\lVert \hat V\right\rVert_{\mathbb H} \sum_{j=1}^d \sqrt{\int_{\mathcal M} \nabla\Xi_j \cdot \nabla\Xi_j \, d\nu} \\
    & \leq \left\lVert \hat V\right\rVert_{\mathbb H} \sum_{j=1}^d \sqrt{-\int_{\mathcal M} (\Delta\Xi_j) \Xi_j \, d\nu} \\
    & \leq \left\lVert \hat V\right\rVert_{\mathbb H} \sum_{j=1}^d \sqrt{\lVert \Delta\Xi_j\rVert_{L^\infty(\nu)} \lVert \Xi_j\rVert_{L^1(\nu)}} \\
    & \leq \left\lVert \hat V\right\rVert_{\mathbb H} \left(\sum_{i=1}^d \lVert \Delta\Xi_i\rVert_{L^\infty(\nu)} \right)^{1/2} \left( \sum_{j=1}^d \lVert \Xi_j\rVert_{L^1(\nu)} \right)^{1/2} \\
    & = \left\lVert \hat V\right\rVert_{\mathbb H} \left\lvert F \circ \Phi^t - F \circ \hat \Phi^t\right\rvert_{2,\infty}\sqrt{\varepsilon(t)}.
\end{align*}
Combining these estimates, we deduce that
\begin{equation}
    \label{eqn:dot_epsilon_ineq}
    \dot\varepsilon(t) \leq \alpha + \beta \sqrt{\varepsilon(t)}.
\end{equation}

Next, since $\varepsilon(0) = 0$ by definition, we can use \cref{lem:prod_log} to integrate~\cref{eqn:epsilon_ineq}, giving
\begin{displaymath}
    \varepsilon(t) \leq \left( \frac{\alpha}{\beta} (W_{-1}(\zeta(t)) + 1) \right)^2, \quad \zeta(t) = - \exp\left(- \frac{\beta^2 t}{\alpha} - 1\right).
\end{displaymath}
Moreover, from \cite[Theorem~1]{Chatzigeorgiou13} we have,
\begin{displaymath}
    -1 - \sqrt{2\tau} -\tau < W_{-1}(-\exp{(-(\tau+1))}) < -1 -\sqrt{2\tau} -\frac{2\tau}{3}, \quad \forall \tau > 0.
\end{displaymath}
With $\tau=\beta^2 t / (2\alpha)$, the lower bound gives
\begin{displaymath}
    -W_{-1}(\zeta(t)) - 1 \leq \sqrt{\frac{\beta^2 t}{\alpha}} + \frac{\beta^2 t}{2\alpha}.
\end{displaymath}
As a result, we obtain
\begin{equation}
    \label{eqn:epsilon_ineq}
    \sqrt{\varepsilon(t)} \leq \sqrt{\alpha t} + \frac{\beta t}{2}.
\end{equation}
The claim of the proposition follows from~\cref{eqn:dot_epsilon_ineq} and~\cref{eqn:epsilon_ineq}. \qed

\bibliographystyle{siamplain}

\end{document}